\documentclass{article}%

\usepackage{graphicx}
\usepackage{amsmath,amssymb,amsthm}
\usepackage{mathrsfs}
\usepackage{xcolor}
\usepackage{hyperref}
\usepackage[numbers,round]{natbib}
\usepackage{bm}
\usepackage{comment}
\usepackage{xcolor}
\usepackage{todonotes}
\usepackage{algorithm,algpseudocode}
\usepackage{tikz}
\usepackage{pgfplots}
\usepgfplotslibrary{polar}
\pgfplotsset{compat=newest}
\usepackage{xpatch}
  \makeatletter
  \xpatchcmd{\@thm}{\fontseries\mddefault\upshape}{}{}{} 
  \makeatother
  
\newtheorem{theorem}{Theorem}[section]
\newtheorem{lemma}[theorem]{Lemma}
\newtheorem{corollary}[theorem]{Corollary}

\newtheorem{remark}[theorem]{Remark}

\newtheorem{assumption}[theorem]{Assumptions}
\newtheorem{example}[theorem]{Example}
\newtheorem{definition}[theorem]{Definition} 

\numberwithin{equation}{section}

\algrenewcommand\algorithmicrequire{\textbf{Input:}}
\algrenewcommand\algorithmicensure{\textbf{Output:}}

\newcommand{\E}{{\mathbb{E}}}

\newcommand{\N}{{\mathbb{N}}}

\newcommand{\R}{{\mathbb{R}}}

\newcommand{\cA}{{\cal A}}
\newcommand{\cB}{{\cal B}}

\newcommand{\cE}{{\cal E}}
\newcommand{\cF}{{\cal F}}

\newcommand{\cL}{{\cal L}}
\newcommand{\cN}{{\cal N}}

\newcommand{\cU}{{\cal U}}

\newcommand{\eps}{{\varepsilon}}

\DeclareMathOperator*{\argmin}{arg\,min} 

\newcommand{\Qs}{{Q^{\ast}}}

\usepackage{bbold}
\usepackage{mathtools}
\newcommand{\tg}{\vartheta}

\newcommand{\wrt}{{with respect to }}

\usepackage{enumitem}
\DeclareMathAlphabet{\pazocal}{OMS}{zplm}{m}{n}

\makeatletter
\newcommand{\vast}{\bBigg@{4}}
\newcommand{\Vast}{\bBigg@{5}}
\makeatother

\title{Importance sampling with unbounded random stopping times: computing committor functions and exit rates without reweighting}

\author{Carsten Hartmann \and Annika Jöster \and Christof Schütte \and Alexander Sikorski \and Marcus Weber}

\begin{document}

\maketitle

\begin{abstract}
    Rare events in molecular dynamics are often related to noise-induced transitions between different macroscopic states (e.g.,~in protein folding). A common feature of these rare transitions is that they happen on timescales that are on average exponentially long compared to the characteristic timescale of the system, with waiting time distributions that have  (sub)exponential tails and infinite support. As a result, sampling such rare events can lead to trajectories that can be become arbitrarily long, with not too low probability, which makes the reweighting of such trajectories a real challenge. 
	Here, we discuss rare event simulation by importance sampling from a variational perspective, with a focus on {the computation of committor functions and mean first exit times that both play a prominent role in molecular dynamics}. The idea is to design importance sampling schemes that (a) reduce the variance of a rare event estimator while controlling the average length of the trajectories and (b) that do not require the reweighting of possibly very long trajectories. In doing so, we study different stochastic control formulations for committor and mean first exit times, which we compare both from a theoretical and a computational point of view, including numerical studies of some benchmark examples.        
\end{abstract}

\tableofcontents

\section{Introduction}

Rare event simulation plays a key role in Scientific Computing, with applications in structural reliability analysis \cite{papaioannou2016sequential}, climate modelling \cite{ragone2018computation}, molecular dynamics \cite{anum}, multienergy systems \cite{christianen2024importance}, or financial risk analysis \cite{liu2005equilibrium}, to mention just a, few examples.   
Since the quantities of interest in rare event simulation (RESIM) are typically extremely small (e.g. probabilities) or extremely large (e.g. waiting times), the key concern of any numerical RESIM algorithm is the control of the relative error by reducing the sample variance of the estimators. 

There are two major classes of variance reduction techniques for RESIM: splitting methods such as RESTART \cite{restart}, Adaptive Multilevel Splitting \cite{ams}, or Subset Simulation \cite{zuev2015subset} that decompose the state space into a collection of nested subspaces, but that are based on the underlying probability distribution, and biasing methods, such as importance sampling \cite{lecuyerIS} that enhance the rare events under consideration by changing the underlying probability distribution and thus altering the rare event's statistics; see \cite{Asmussen2013,junejaHandbook} for an overview. 
An advantage of splitting methods is that they are non-invasive and relatively easy to parallelize, a disadvantage is that they typically require some prior knowledge of a low-dimensional reaction coordinate that allows to monitor the rare event. In contrast, biasing techniques like importance sampling are mostly intrusive because they require to bias the underlying model dynamics or sampling mechanism. Exceptions are black-box importance sampling methods that change only the distribution of model input parameters, e.g. \cite{liu2005equilibrium,pan2020adaptive}, or  asymptotic techniques that approximate rare event probabilities without sampling from the large deviations rate function of the problem, e.g. \cite{grafke2019numerical,tong2021extreme}. 
We should also mention sequential Monte Carlo methods that combine both worlds, in that they embed change of measure techniques into a splitting-like framework, e.g. \cite{cerou2012,dai2022invitation}. The same goes for resampling or rescaling methods (e.g.~\cite{bugallo2017adaptive,elvira2023gradient}) that use previously sampled values of the quantity of interest to adaptively change the proposal distribution, either globally on in certain subsets of the state space. 
{Such techniques are also relevant in the context of Bayesian inverse problems, especially particle filtering (e.g.~\cite{crisan2002survey}), and several authors have recently embedded control-based importance sampling into a particle filtering framework, e.g. \cite{Reich2018,rached2026importance}; see also \cite{ben2024double} for an approximation using the corresponding stochastic optimal control problem using systems of interacting particles.}

In this paper we discuss rare event simulation for stochastic differential equations from a variational perspective, with a focus on  {the computation of exit rates and committor probabilities (i.e. the probability to reach one set before another) that both play a prominent role in molecular synamics}. Our approach is partly in the spirit of the adaptive importance sampling (AIS) technique developed by Dupuis, Wang and co-workers \cite{DupuisWang2004,DupuisWang2007} that has been adapted to the diffusion setting in \cite{Weare2012,dupuis2012importance,Kostas2015}. 
The main similarity is that our approach uses feedback control representations of the optimal change of measure that is adapted to the system state; cf.~\cite{Hartmann2012,HartmannEnt2017}. 

There are, however, a few key differences: Firstly, our approach is nonasymptotic in that it does not rely on large deviations asymptotics, such as small noise asymptotics, large particle number limits or large waiting times. Secondly, and most importantly, the actual rare event estimation does not require reweighting, because the quantities of interest can be estimated directly from the value function of the associated stochastic control problem. While this will in general lead only to a biased estimate of the quantity of interest, which is related to the value function by a nonlinear transformation, it can be beneficial if the reweighting is sensitive to bad approximations of the optimal change of measure (e.g. in  high dimensions or if simulation over long time horizons is required). 

\paragraph{Contribution of this work} 

We describe a general approach that represents the optimal importance sampling measure that minimizes the variance of the estimator by a convex transformation of the underlying random variable. The quantity of interest and the transformed quantity are related by an inequality that can be turned into an equality by a suitable change of measure. We show that the change of measure for which equality is attained is the optimal importance sampling distribution. Specifically, we consider (1) a logarithmic  transformation of the moment generating function of the random variable (with the exponential function being the convex transformation), and (2) a square root transformation of the second moment (with the quadratic function as the convex transformation). The corresponding random variables are path functionals of SDEs that are defined up to an unbounded random stopping time. We study the associated indefinite time horizon stochastic control problems in detail that are (1) of linear quadratic type and (2) of risk-sensitive form. 
To our knowledge, AIS of problems that involve unbounded random stopping times is not well represented in the literature, {despite its relevance in statistical mechanics and molecular dynamics (see, e.g. \cite{donati2017girsanov,yuan2024optimal,du2026rare}).} 

{We look at this phenomenon in more detail, and we show that the two stochastic control formulations studied in this paper, while they both characterize a (theoretical) zero-variance importance measure, lead to importance sampling strategies with vastly different numerical costs. The reason is that, depending on the formulation, the likelihood of the rare event can increase or decrease, as measured by the average length of the resulting controlled trajectories. Roughly speaking, the two schemes under consideration have adverse effects on the tails of the underlying path distribution as they control different moments of the distribution: the formulation based on the logarithmic transformation of the moment generating function implies a bound on \emph{all} moments (including the average length of the controlled trajectories), while the formulation based on the second moment does not naturally give rise to a bound on the first moment under the controlled dynamics and thus does not necessarily reduce the average length of the trajectories.}
We discuss pathological cases in form of optimal controls that generate a zero-variance change of measure, but lead to sample trajectories of infinite length with probability one. Such pathologies have been described in the seminal paper \cite{awad2013}, and we can now provide a systematic control interpretation of these observations {from both a theoretical and an algorithmic perspective, generalizing our own works \cite{hartmann2024riskneutral,anum} that mainly focused on the theorical aspects.} 

The feedback control policies in the random stopping time scenario are stationary (i.e. without explicit time dependence, assuming that the processes are time-homogeneous). We devise an approximate policy iteration (API) scheme and prove convergence to the optimal control policy and the value function. Following ideas in \cite{chang1986successive}, we prove that the API scheme in the log transform case is unconditionally convergent, whereas the square root case requires some regularization to bound the controls during the policy evaluation steps. {As we demonstrate, the regularization can lead to a dramatically low rate of convergence of the API scheme, which is yet another indication that AIS strategies for unbounded stopping times that are solely based on minimization of the second moment should be handled with care.}
The API algorithms are tested for a benchmark committor problem, confirming the theoretical predictions with regard to monotonicity of the control value and convergence to the optimal policy. Despite being relatively simple (yet high-dimensional), the numerical examples show some features that are relevant when the approach is applied to committor function computations for more complicated dynamical system, such as biomolecular systems. {We also briefly discuss the challenges in connection with molecular dynamics on phase space.}

\paragraph{Outline of the paper}
 The rest of the article is structured as follows: In Section \ref{sec:zerovar} we outline the idea of using certainty-equivalence principles with strictly convex transformations to characterize zero-variance importance measures, which is then spelled out for stochastic differential equations in Section \ref{sec:SOC}, in which the associated {stochastic optimal control (SOC)} representations of the optimal change of measure are derived. Section \ref{sec:API} is devoted to the formulation and the analysis of the API algorithm that is tested for simple benchmark committor problems in Section \ref{sec:numerics}. {Control strategies for molecular dynamics problems and the notoriously difficult exit problem are discussed in Section \ref{sec:discussion}, including a link between AIS and control variates. The findings are summarized in Section \ref{sec:fin}. Appendix \ref{sec:genSOC} records some theoretical results that provide the control-theoretic background for Section \ref{sec:SOC}, Appendix \ref{app:hfLimit} discusses the high friction approximation of the controlled underdamped Langevin dynamics that is relevant for the integration of the API algorithm into molecular dynamics codes.}

\section{Zero-variance change of measure}\label{sec:zerovar}

The key ingredient of importance sampling (IS) is a change of the underlying probability measure that reduces the estimation variance. To explain the key idea, we consider a nonnegative random variable $S\geq 0$ on some probability space $(\Omega,\cE,P)$. Suppose we want to estimate the expectation $\E[S]$ of $S$ under the probability $P$. Further, we suppose that there exists another probability measure $Q$ that has a strictly positive density $L$ \wrt $P$, at least when restricted to $\{S>0\}$. This implies that $P$ and $Q$ are mutually absolutely continuous on the set $\{S\neq 0\}$, and it allows us to recast $\E[S]$ as (see \cite{awad2013})
\begin{equation}\label{com}
	\E[S]=
    \E_Q\!\left[S L^{-1}\right]\,,\quad L:=\frac{dQ}{dP}\,.
\end{equation}
Our aim is to choose $Q$ such that the variance of $SL^{-1}$ is minimal under $Q$. For nonnegative random variables that we consider here, even zero variance is theoretically possible, but the variance-minimizing measure $\Qs$ necessarily depends on $\E[S]$. Hence, direct sampling from $\Qs$ is not feasible.     

Here we will characterize the zero-variance property by a convexity argument that resembles what is known as certainty-equivalence principle. In control theory, certainty-equivalence means that the optimal control law for a stochastic dynamics can be recast as an optimal control law for an associated deterministic (certainty equivalent) problem; see \cite[Sec.~1]{whittle2002} and references therein.  
Here the idea is as follows: In \eqref{com}, we replace the random variable $S$ by a transformed random variable $\varphi(S)$ where $\varphi$ is a strictly convex (strictly increasing or decreasing) function with inverse $\varphi^{-1}$. The inverse transformation is used to invert the transformation $\varphi$ \emph{after} taking the expectation and so returning to the physical scale of the original random variable $S$. 
Thus, instead of $\E[S]$, we consider the certainty-equivalent expectation 
\begin{equation}
    \varphi^{-1}(\E[\varphi(S)])\,.
\end{equation} 
Two notable special cases are
	\begin{enumerate}
		\item $\varphi(s)=e^{-\lambda s}$ for $\lambda>0$, with the property
		\begin{equation}
		-\lambda^{-1}\log\E\big[e^{-\lambda S}\big] \le  \E[S]\,
		\end{equation}
        \item $\varphi(s)=|s|^p$ for $p > 1$, with the property
		\begin{equation}
		\left(\E\big[S^p\big]\right)^{1/p} \ge \E[S]\,.   
		\end{equation}
	\end{enumerate}
Since $\varphi$ is \emph{strictly} convex, equality in both cases holds iff $S$ is almost surely constant (in other words: deterministic), and we can use this fact as a characterization of a change of measure that nullifies the variance, since a random variable is constant iff its variance is zero. 

We will now discuss the two aforementioned special cases that both give rise to computationally feasible expressions for the zero-variance change of measure.

\subsection{First approach: moment generating function} 

Firstly, we suppose that $L>0$ and assume the quantity of interest to have the form of a moment generating function (MGF)
\begin{equation}
\E\!\left[e^{-\lambda S} \right]=\E_Q\!\left[e^{-\lambda S}L^{-1} \right].
\end{equation}
Equivalently,
\begin{equation}
\E\!\left[e^{-\lambda S} \right] = \E_Q\!\left[e^{-\lambda (S - \lambda^{-1}\log L)} \right] .
\end{equation}
\begin{lemma}[Gibbs variational principle, cf.~\cite{daipra1996,anum}]\label{Gibbs}
	Under suitable conditions guaranteeing that expressions remain finite, we have 
	\begin{equation}\label{gibbs}
	-\lambda^{-1}\log \E\left[e^{-\lambda S} \right] =\inf_{Q\ll P}\left\{\E_Q(S)+\lambda^{-1}D(Q\vert P)\right\},
	\end{equation}
	where
\begin{equation}
D(Q\vert P):= \begin{cases}
    \E_Q[\log L] & \text{if }\; Q\ll P\,,\; \log L\in L^1_Q\\ +\infty &  \textrm{else}
\end{cases}
\end{equation}
is the Kullback-Leibler divergence or relative entropy between $Q$ and $P$. The probability measure $\Qs$, for which the infimum in (\ref{gibbs}) is attained is given by
\begin{equation}
	\frac{d\Qs}{dP}=\frac{e^{-\lambda S}}{\E\left[e^{-\lambda S} \right]}\,.
\end{equation}
Moreover, if $P\ll \Qs$, 
\begin{equation}\mathrm{Var}_{\Qs}\!\left(e^{-\lambda S}L^{-1}_* \right)=0,\quad L_*^{-1}:=\frac{dP}{d\Qs}.\end{equation}
\end{lemma}

The optimal change of measure inevitably depends on the quantity of interest. Yet, as we will see in Section \ref{sec:SOC} below, the variational formulation gives rise to a computationally feasible optimization problem.

\subsection{Second approach: second moment minimization} 

We now consider quantities in form of second moments: another way of minimizing the variance of $SL^{-1}$ under $Q$ is based on the observation that 
\begin{equation}
\mathrm{Var}_Q(SL^{-1})\geq 0 \quad\Longleftrightarrow\quad  \E_Q[S^2L^{-2}]\ge \left(\E[S]\right)^2\,, \end{equation}
since $\E[S]=\E_Q[SL^{-1}]$ for all $Q\sim P$. Let us introduce the following expressions.
\begin{definition}\label{def:squareroottransform}
	We define 
	\begin{equation}
			\Psi := \E[S]\,,\quad V:=\inf\limits_{Q}\E_Q\left[S^2L^{-2}\right]\,,\quad 
	\end{equation}
    where we assume throughout that $Q\sim P$ on the set $\{S\neq 0\}$.
\end{definition}

\begin{lemma}[Zero variance]\label{lemma_2ndMoment}
	The expressions $V$ and $\Psi$ satisfy
\begin{equation}
\Psi^2 \leq V,
\end{equation}
with equality iff $Q=\Qs$ where 
\begin{equation}
	\frac{d\Qs}{dP}:= \frac{S}{\E[S]}\quad \Qs\text{-a.s.}\,.
\end{equation}
In this case, $V=E_\Qs\!\left[S^2L_*^{-2}\right]$, with $L_*^{-1}=\frac{dP}{d\Qs}$, in other words,
\begin{equation}
		\mathrm{Var}_\Qs(SL_*^{-1}) = 0\,.
\end{equation}
\begin{proof}
        For $S\notin L^2_P$ there is nothing to prove, so we suppose that $S\in L^2_P$ or, equivalently, $SL^{-1}\in L^2_Q$.
		By Jensen's inequality, 
		\begin{equation}
		\E_Q\!\left[\left(S L^{-1}\right)^2\right]\geq \left(\E_Q\left[S L^{-1}\right]\right)^2=\left(\E[S]\right)^2=\Psi^2,
		\end{equation}
		hence,
		\begin{equation}
			V=\inf\limits_{Q}\E_Q\!\left[S^2L^{-2}\right]\geq \Psi^2\,,
		\end{equation}
        in particular, 
        \begin{equation}
			\E_\Qs\!\left[S^2L_*^{-2}\right]\geq \Psi^2\,,
		\end{equation}
        To show the reverse inequality, let $S_n=S\mathbf{1}_{\{S< n\}}$ and define 
        \begin{equation}
			L_n=\frac{dQ_n}{dP}= \frac{S_n}{\E[S_n]}\,.
		\end{equation}
        Clearly, $Q_n(0<S<n)=1$, and so 
        \begin{equation}
            \E_{Q_n}\!\left[S^2L_n^{-2}\right] = \int_{\{0<S<n\}} \left(\E[S_n]\right)^2 dQ_n = \left(\E[S_n]\right)^2
        \end{equation}
        and Fatou's Lemma implies 
        \begin{equation}
         \E_\Qs\!\left[S^2L_*^{-2}\right] \le \liminf_{n\to\infty}\E_{Q_n^*}\!\left[S^2L_n^{-2}\right] = \lim_{n\to\infty}\left(\E[S_n]\right)^2 = \left(\E[S]\right)^2\,.
        \end{equation}
        Since the rightmost expression equals $\Psi^2$, it follows that  
        \begin{equation}
        V  = \E_\Qs\!\left[S^2L_*^{-2}\right] = \Psi^2\,,
        \end{equation}
        which concludes the proof.
	\end{proof}
\end{lemma}

In the next section, we will derive stochastic control formulations for the aforementioned variational principles. Specifically, we consider random variables $S$ that are path functionals of solutions to stochastic differential equations.

\section{Stochastic Optimal Control}\label{sec:SOC} 

In what follows, we consider $X_s\in\R^d$ being the solution of the stochastic differential equation (SDE) 
\begin{equation}\label{sde0}
	dX_s=b(X_s)ds+\sigma(X_s)dB_s,\quad X_0=x\,,
\end{equation}
with $b\colon\R^d\to\R^d,\,\sigma\colon \R^d\to\R^{d\times k}$ satisfying the usual Lipschitz and growth conditions that guarantee that (\ref{sde0}) has a unique strong solution. Here and in the following, $B_s\in\R^k$ denotes standard Brownian motion where $k\le d$. We call $\cL$ the second order differential operator given by 
\begin{equation}
    \cL = \frac{1}{2}\sigma\sigma^\top\colon\nabla^2 + b\cdot\nabla\,.
\end{equation}
Mostly, we will deal with gradient drift $b=-\nabla V$ for some smooth potential $V\colon\R^d\to\R$ that is bounded from below and constant diffusion coefficient $\sigma=\sqrt{2 \vartheta}$ for some parameter $\vartheta>0$; in this case, $k=d$.  

We consider $P$ to be a probability measure of the path space $C([0,\infty),\R^d)$ associated with (\ref{sde0}), while the measure $Q$ will be associated with a controlled SDE of the form (precise definitions will be given below)
\begin{equation}\label{sde+u}
	dX^u_s=(b(X^u_s)+\sigma(X^u_s)u_s)ds+\sigma(X^u_s)dB_s,\quad X^u_0=x,
\end{equation}
where  $u_s= \underline{u}(X_s^u)$ is some stationary feedback control to be specified below. For each control $u$, we will also consider the corresponding SDE for $-u$, i.e.,
\begin{equation}\label{sde-u}
	dY^u_s=(b(Y^u_s)-\sigma(Y^u_s)u_s)ds+\sigma(Y^u_s)dB_s,\quad Y^u_0=x,
\end{equation}
where $\tilde{u}_r:=-u_r= -\underline{u}(Y_r^u)$. We assume throughout that the controls are such that the corresponding SDEs have unique strong solutions.

In what follows, we will consider stopped versions of the controlled and uncontrolled SDEs that are defined up to some stopping time 
\begin{equation}\tau^u:=\inf\{t\geq0\vert X^u_t\notin D\} \quad \text{and} \quad\tilde{\tau}^u:=\inf\{t\geq0\vert Y^u_t\notin D\}
\end{equation} 
for some measurable bounded open set $D\subset\R^d$ where we use the shorthand $\tau=\tau^{u=0}$. We will sometimes simply write $\tau$ instead of $\tau^u$ or $\tilde{\tau}^u$ if it is clear from the context that $\tau$ is a stopping time \wrt any of the controlled processes $X^u$ or $Y^u$.
{Unless otherwise stated, it is assumed that all such stopping times are almost surely finite. This can be, for example, guaranteed by assuming that the dynamics is weakly controllable, for which it is sufficient that the drift is mean reverting and the diffusion coefficient is non-denegerate; see, e.g. page \pageref{ass:API} for precise assumptions.} 
 Since we will extensively use changes of measures, we first state Girsanov's Theorem (cf.~\cite[Thm.~38.5]{rogers2000diffusions})

\begin{theorem}[Girsanov's Theorem]\label{thm:girsanov}
Let $u=(u_s)_{s\ge 0}$ be an admissible control, such that 
\begin{equation}
\mathscr{L}_t = \exp\!\left(- \frac{1}{2}\int_{0}^{t}\vert u_s\vert^2 ds + \int_{0}^{t} u_s dB_s \right)
\end{equation}
is a uniformly integrable martingale \wrt $P$. Then the path space measure $P^u$ defined by 
\begin{equation}
\frac{dP^u}{dP} = \mathscr{L}_\infty
\end{equation}
is equivalent to $P$, and the $P$-law of $X^u$ is the $P^u$-law of $X$. In other words, 
\begin{equation}
B_t^u = B_t - \int_0^t u_s\,ds\,,\quad t > 0
\end{equation}
is a Brownian motion under $P^u$, and $P^u$-a.s.
\begin{equation}
	dX_s=(b(X_s)+\sigma(X_s)u_s)ds+\sigma(X_s)dB^u_s,\quad X_0=x\,.
\end{equation}
\end{theorem} 
	
\begin{remark}\label{rem:reweighting}
    The process $\mathscr{L}=(\mathscr{L}_t)_{t\ge 0}$ is the SDE version of the abstract likelihood ratio $L=\frac{dQ}{dP}$ from the previous section. Girsanov's Theorem implies that for any $P$-integrable and $\mathcal{F}_\tau$-measurable random variable $S=S(X)$, where $\mathcal{F}_\tau$ is the $\sigma$-algebra generated by $(X_s)_{0\le s\le \tau}$, we have 
    \begin{align*}
        \E^x[S(X)] & = \E^x_{P^u}\!\left[S(X)\mathscr{L}_\infty^{-1}(X)\right]\\ & 
        = \E^x_{P^u}\!\left[\E^x_{P^u}\!\left[S(X)\mathscr{L}_\infty^{-1}(X)|\mathcal{F}_\tau\right]\right]\\
        & = \E^x_{P^u}\!\left[S(X)\E^x_{P^u}\!\left[\mathscr{L}_\infty^{-1}(X)|\mathcal{F}_\tau\right]\right]\\
        & = \E^x_{P^u}\!\left[S(X)\mathscr{L}_\tau^{-1}(X)\right]\\
        & = \E^x\!\left[S(X^u)\mathscr{L}_\tau^{-1}(X^u)\right]\,.
    \end{align*}
\end{remark}
\paragraph*{Stochastic control problem with indefinite time horizon.}

Before we come to the stochastic control representations of the zero-variance change of measure, we introduce a general stochastic control problem with a random, unbounded stopping time. To this end, we define the cost functional 
\begin{equation}\label{cost}
		J(x,u)=\E^{x}\left(\int_{0}^{\tau^u} f(X^u_r,u_r)\Gamma(r)dr+g(X^u_{\tau^u})\Gamma(\tau^u)\right),
\end{equation}	
where $X^u$ is the solution to (\ref{sde+u}), and 
\begin{equation}\label{costGamma}
\Gamma(s):=\exp\!\left(-\int_{0}^{s}\beta(X^u_r,u_r)dr\right)\,.
\end{equation}
The next lemma states necessary optimality conditions for the minimization of the objective function (\ref{cost})--(\ref{costGamma}) under the controlled dynamics (\ref{sde+u}).

\begin{lemma}[Generalized stochastic optimal control problem]\label{generalized_SOC}
Let 
\begin{equation}
v(x):=\min_{u\in \cA} J(x,u)
\end{equation}
be the value function associated with (\ref{cost})--(\ref{costGamma}) where $\cA$ is a set of admissible Markovian controls with values in $U\subset\R^k$ such that (\ref{sde+u}) has a unique strong solution. 
Then $v=v(x)$ solves 
the Hamilton-Jacobi-Bellman (HJB) equation 
    \begin{equation}\label{HJB2}
        \begin{aligned}
        \min_{a\in U}\Bigg \{ f-\beta\cdot v+(b+\sigma a)\cdot \nabla v + \frac{1}{2}\sigma\sigma^\top:\nabla_{xx}v\Bigg\} & = 0\quad \text{in $D$}\\
        v & = g \quad \text{on $\partial D$}\,.
        \end{aligned}
    \end{equation}
\end{lemma}

\begin{proof}
    See Appendix \ref{sec:genSOC}.
\end{proof}

We stress that a general formulation would involve a more general dependence of drift and diffusion coefficients on the control, or explicit time-dependence, but we confine ourselves to an SDE of the form (\ref{sde+u}) because it serves our purposes and a more general form would come with stringent regularity assumptions (especially when the diffusion coefficient is controlled) that we want to avoid here.

\subsection{Moment generating function}

We define 
\begin{equation}\label{pathfunctional}
    W(X):=\int_{0}^{\tau}f(X_s)ds+ g(X_{\tau})\,.
\end{equation}
and revisit the certainty-equivalent expectation of Lemma \ref{Gibbs} that was based on a logarithmic transformation and the associated inequality
\begin{equation}
-\lambda^{-1}\log\E^x\big[e^{-\lambda W(X)}\big]\le \E^x[W(X)]\,,
\end{equation}
that follows from the convexity of the exponential function.  
To this end, we set $\beta \equiv 0$ and $f(x,a):=f(x)+\frac{1}{2\lambda}|a|^2$ for some $\lambda>0$ in Lemma \ref{generalized_SOC}, which leads to a standard {SOC} problem that is linear-quadratic in the controls, with a possibly nonlinear state dependence; see \cite[Sec.~VI.5]{fleming2006} for details.  

\begin{definition}[Stochastic optimal control problem no.~1]\label{SOC1}
Minimize
\begin{equation}
        J_1(x,u)= 
        \E^{x}\left[\int_{0}^{\tau^u} \left(\frac{1}{2\lambda}\vert u_s\vert^2 +f\left(X_s^{u} \right)\right)ds +g\left(X_{\tau^u} \right) \right]
\end{equation}
subject to 
\begin{equation}
        dX^u_s=(b(X^u_s)+\sigma(X^u_s)u_s)ds+\sigma(X^u_s)dB_s,\quad X^u_0=x
\end{equation}
\end{definition}

It can be shown, using properties of conditional expectations, that the corresponding value function $v_1(x):=\min_{u\in\cA} J_1(x,u),\,x\in\overline{D}$ solves the HJB equation (see, e.g.~\cite[Sec.~VI.5]{fleming2006})
\begin{equation}\label{HJB1}
	\begin{cases}
		&\min\limits_{c\in U}\left \{ f+\frac{1}{2\lambda }\vert c\vert^2 +\cL v_1 +\sigma^\top \nabla v_1\cdot  c\right\}=0, \quad x\in D\\ 
		& v_1(x)=g(x),\quad x\in \partial D\,.
	\end{cases}
\end{equation}
Using that the unique minimizer given by
\begin{equation}\label{OC1}
c_1^*=-\lambda \sigma^\top \nabla v_1, 
\end{equation}
the HJB equation (\ref{HJB1}) is equivalent to  
\begin{equation}\label{PDE1}
	\begin{cases}
		& f+\cL v_1 -\frac{\lambda}{2}\left|\sigma^\top \nabla  v_1\right|^2 =0, \quad x\in D\\ 
		& v_1(x)=g(x),\quad x\in \partial D.
	\end{cases}
\end{equation}

We will argue that the probability measure $P^{u^*}$ that is induced by the optimal control $u^*_s= c_1^*\left(X^{u^*}_s \right)$ with $c_1^*$ as given by (\ref{OC1}), agrees with the probability measure $\Qs$ from Lemma \ref{Gibbs} that is given by 
\begin{equation}
\frac{d\Qs}{dP}=\frac{e^{-\lambda S}}{\E\left[ e^{-\lambda S}\right]},
\end{equation}
with $S=W(X)$ as defined by (\ref{pathfunctional}). This connection is expressed by the next Lemma that is the stopping time analogue of the famous Bou\'e-Dupuis formula~\cite{BD98}; see also \cite{daipra1996,anum}.

\begin{lemma}[Value function and Gibbs variational principle]\label{lem:socGibbs}
	Recall $\frac{dP^u}{dP}=\mathscr{L}_\infty$ and let
    \begin{align*}
		J_1(x,u)=\E^x\!\left[W(X^u)+\lambda^{-1}\log\mathscr{L}_\tau(X^u)\right]
	\end{align*}
    where we use the notation $\mathscr{L}_\tau(X^u)$ to indicate that $u_s= \underline{u}(X_s^u)$ is a feedback control and $\tau=\tau^u$ is a hitting time for $X^u$. Equivalently,  
\begin{align*}
		J_1(x,u)=\E^x_{P^u}\!\left[W (X)+\lambda^{-1}\log\mathscr{L}_\tau(X)\right],
	\end{align*}
Then 
	\begin{equation}
		v_1(x)=	- \lambda^{-1}\log \E^x\!\left[e^{-\lambda W(X)}\right] ,
	\end{equation}
	and	the optimal change of measure is given by
	\begin{equation}
		\frac{dP^{u^*}}{dP}=\frac{e^{-\lambda W(X)}}{\E^x\!\left[e^{-\lambda W(X)}\right]}.
	\end{equation}
    \end{lemma}
	\begin{proof}
		By Jensen's equality and Remark \ref{rem:reweighting}, 
		\begin{align*}
			J_1(x,u)&=\E^x_{P^u}\!\left[W (X)+\lambda^{-1}\log\mathscr{L}_\tau(X)\right]\\ 
            & =-\E^x_{P^u}\!\left[\lambda^{-1}\log\left(e^{-\lambda W(X)}\mathscr{L}^{-1}_\tau(X)\right)\right]\\
			&\geq -\lambda^{-1}\log\E^x_{P^u}\!\left[e^{-\lambda W(X)} \mathscr{L}^{-1}_\tau(X)\right]\\
			&=-\lambda^{-1}\log \E^x\!\left[e^{-\lambda W(X)}\right],
		\end{align*}
		thus
		\begin{align*}
			\inf_{u\in\cA} J_1(x,u)=v_1(x)\geq -\lambda^{-1}\log \E^x\!\left[e^{-\lambda W(X)}\right], 
		\end{align*}
		with equality iff
		\begin{align*}
			e^{-\lambda W(X)} \mathscr{L}^{-1}_\tau(X) =	\E^x_{P^u}\!\left[e^{-\lambda W(X)} \mathscr{L}^{-1}_\tau(X)\right]=\E^x\!\left[e^{-\lambda W(X)} \right],\quad P^u\text{-a.s.}\,.
		\end{align*}
		The first equality holds iff $P^u=P^{u^*}$, with 
		\begin{align}
			\frac{dP^{u^*}}{dP}=\frac{e^{-\lambda W(X)}}{\E^x\!\left[e^{-\lambda W(X)}\right]},\quad \quad P^{u^*}-\text{a.s.}
		\end{align}
		In this case, $v_1(x)=-\lambda^{-1}\log\E^x\!\left[e^{-\lambda W(X)}\right]$.
	\end{proof}

Lemma \ref{lem:socGibbs} shows that the certainty-equivalence principle of Lemma \ref{Gibbs} has a straightforward formulation as a stochastic optimal control problem.  

\subsection{Second moment minimization}

We now come to the certainty-equivalent expectation of Definition \ref{def:squareroottransform} that was based on a square root transformation and the associated inequality
\begin{equation}
\sqrt{\E^x\big[W(X)^2\big]} \ge \E^x[W(X)]\,.
\end{equation}
To this end, we set $\beta(X^u_s, u_s):=-u_s^2$ and $f\equiv 0$ in Lemma \ref{generalized_SOC}, which leads to the following risk-sensitive SOC problem; see \cite{nagai1996bellman} for details.

\begin{definition}[Stochastic optimal control problem no.~2]\label{SOC2}
	Minimize
	\begin{equation}\label{quadraticcost}
		J_2(x,u):=\E^{x}\left[ g^2\left(Y^u_{\tau} \right) \exp\left(\int_{0}^{\tau}u^2_s ds\right) \right]
	\end{equation}
	subject to the controlled SDE
	\begin{equation}
		dY^u_s=\left(b(Y^u_s)-\sigma(Y^u_s)u_s\right)ds+ \sigma(Y^u_s)dW_s,\quad Y^u_0=x\,.
	\end{equation}	
    Here and in what follows, we write $\tau=\tilde{\tau}^u$ for the stopping time under the controlled process $Y^u$.
\end{definition}
    
It follows from the dynamic programming principle (e.g.~\cite{nagai1996bellman}) that the value function
	\begin{equation}
		v_2(x)=\min_{u\in\cA} J_2(x,u)
	\end{equation}
solves the risk-sensitive HJB equation
	\begin{equation}\label{HJB_2}
		\begin{cases}
			&\min\limits_{c\in U} \{\vert c\vert^2 v +(b-\sigma c)\cdot \nabla  v +\frac{1}{2} \sigma \sigma^\top \colon\nabla^2 v\}=0,\quad x\in D\\
			& v(x)=g^2(x),\quad x\in \partial D\,.
		\end{cases}
	\end{equation}
The optimal control is given by $u^*_s = c_2^*(Y_s^{u^*})$ where  
\begin{equation}\label{eq:c2star}
c_2^*(x)=\frac{\sigma^\top \nabla v_2(x)}{2v_2(x)}=\frac{1}{2} \sigma^\top \nabla \log v_2(x).
\end{equation}
is the unique minimizer in (\ref{HJB_2}). As a consequence, the HJB equation can be recast as the equivalent boundary value problem 
	\begin{equation}\label{PDE_2}
	\begin{cases}
		& \cL v-\frac{1}{4v} \left|\sigma^\top \nabla v\right|^2 =0,\quad x\in D\\
		& v(x)=g^2(x),\quad x\in \partial D
	\end{cases}
\end{equation}
that is well-posed if $v>0$ in the interior of the domain $D$. The next lemma is the stopping time analogue of Lemma 4.3 in \cite{dupuis2012importance} that  connects minimization of the second moment with the risk-sensitive criterion of Definition \ref{SOC2}.

\begin{lemma}\label{quadratic}
It holds  
	\begin{equation}
		\E_{P^u}^x\!\left[W^2(X) \mathscr{L}^{-2}_\tau(X) \right] = \E_{\tilde{P}^u}^x\!\left[W^2(X)\exp\!\left(\int\limits_{0}^{\tau} \vert u_r\vert^2 dr\right)\right] 
	\end{equation}	
    where $\frac{d\tilde{P}^u}{dP}=\exp\!\left(-\frac{1}{2}\int_{0}^{\infty}\vert u_r\vert^2 dr -\int_{0}^{\infty} u_r dB_r\right)$. Equivalently,
    \begin{equation}
		\E_{P^u}^x\!\left[W^2(X) \mathscr{L}^{-2}_\tau(X) \right] =\E^x\!\left[S^2(Y^u)\exp\!\left(\int_{0}^{\tau} \vert \underline{u}(Y_r^u)\vert^2 dr\right) \right],
	\end{equation}	
    with $\tau=\tilde{\tau}^u$ in the rightmost expression denoting the stopping time under $Y^u$.
	\end{lemma} 
    
	\begin{proof}
		By Girsanov's theorem,  the $P$-law of $X^u$ is the $P^u$-law of $X$ and the $P$-law of $Y^u$ is the $\tilde{P}^u$-law of $X$, where
		\begin{equation}
				\frac{dP^u}{dP}=\exp\!\left(-\frac{1}{2}\int_{0}^{\infty}\vert u_r\vert^2 dr +\int_{0}^{\infty} u_r dB_r\right)
		\end{equation}
        and 
        \begin{equation}
				\frac{d\tilde{P}^u}{dP}=\exp\!\left(-\frac{1}{2}\int_{0}^{\infty}\vert u_r\vert^2 dr -\int_{0}^{\infty} u_r dB_r\right).
		\end{equation}
        Then, using the Radon-Nikodym chain rule with $\mathscr{L}^{-1}_\infty=\frac{dP}{dP^u}$, 
		\begin{align*}
				\E_{P^u}^x\!\left[W^2(X) \mathscr{L}^{-2}_\infty(X) \right] & =\E_{\tilde{P}^u}^x\!\left[W^2(X) \mathscr{L}^{-1}_\infty(X)\frac{dP}{d\tilde{P}^u}(X) \right] \\ 
                & =\E^x\! \left[W^2(Y^u)\exp\!\left(\int_{0}^{\tau} \vert \underline{u}(Y_r^u)\vert^2 dr\right) \right].
		\end{align*}	
        which proves the assertion. 
	\end{proof}
    Lemma \ref{quadratic} establishes a relation between minimization of the second moment and the risk-sensitive criterion of Definition \ref{SOC2}. 
    As a consequence of Lemma \ref{lemma_2ndMoment} and \cite[Thm.~4]{awad2013}, the minimum in Definition \ref{SOC2} is attained under the optimal control $u^*$ and the dynamics  (\ref{sde-u}), such that 
    \begin{align*}
    \E^x[g(X)]^2 =v_2(x) = \E^{x}\left[g^2\left(Y^{u^*}_{\tau} \right) \exp\left(\int_{0}^{\tau}|u^*_s|^2 ds\right) \right].
    \end{align*}

{
\begin{remark}
    In the finite time case (i.e.~when $\tau=T$ is deterministic and $\cL$ is replaced by $\frac{\partial}{\partial t}+\cL$),  the function $\sqrt{v_2}$ solves a Kolmogorov backward equation. By construction, its solution equals the first moment of $g$, conditional on the initial value $X_t=x$. As a consequence, the square root of the solution to the risk-sensitive HJB equation for the second moment equals the first moment squared, which is in another indication of the zero-variance property; see \cite{awad2013,talayBook}.
\end{remark}
}
    
\subsection{Transformations of boundary value problems}

Assuming sufficient regularity, the nonlinear dynamic programming equations (\ref{HJB1}) and (\ref{HJB_2}) associated with the SOC problems of Definitions \ref{SOC1} and \ref{SOC2} have a direct interpretation in terms of nonlinear transformations of a linear boundary value problem. To see this, let $\Psi\in C^2(D)\cap C(\overline{D})$ be the classical and nonnegative solution of the linear boundary value problem 
\begin{equation}\label{linBVP}
		\begin{cases}
			&\cL \Psi-F\Psi =-H,\quad \text{in}~D\\
			& \Psi= G,\quad\text{on}~\partial D,
		\end{cases}
	\end{equation}
for some sufficiently regular functions $F,G,H\ge 0$ where $\cL$ is the second-order differential operator associated with the uncontrolled SDE (\ref{sde0}). 

By the Feynman-Kac Theorem (see, e.g.~\cite[Thm.~1.3.17]{pham2009continuous} or \cite[Ex.~9.12]{oeksendal2003}), the solution to (\ref{linBVP}) is of the form
\begin{equation}
	\Psi(x) = \E^x\!\left[e^{-\int_0^\tau F(X_s)}G(X_\tau) + \int_0^\tau e^{-\int_0^r F(X_s)}H(X_r)\,dr\right].
\end{equation}

Our aim now is to replace $\Psi$ by one of its certainty-equivalent expressions and derive the corresponding boundary value problem, following the line of thought of Section \ref{sec:zerovar}. To this end,  let $\cL^a \Phi=\cL \Phi + a\cdot \sigma^\top \nabla \Phi$ for some $a\in \R^k$, and consider a smooth function $v=v(x)$ that is related to $\Psi=\Psi(x)$ by 
\begin{equation}
\Psi(x):=(\tg\circ v)(x)
\end{equation}
for some invertible transformation $\tg$. The next Lemma explains how the transformation $\tg$ is related to a transformation between uncontrolled dynamics with generator $\cL$ and controlled dynamics with generator $\cL^a$.    

\begin{lemma}\label{lem:transform}
Let $\Psi=\tg\circ v$ for an invertible mapping $\tg$, such that \begin{equation}\frac{\tg''\circ v}{\tg' \circ v}<0\,, \quad\tg'\neq 0\,.\end{equation} 
Then
	\begin{align}
	\cL \Psi =\left(\tg' \circ v \right) \cdot
	\min\limits_{a\in  \R^k} \left\{ \cL^a v-\frac{\tg' \circ v}{2 \tg'' \circ v}\vert a \vert^2  \right\}.	
	\end{align}
\end{lemma}
    \begin{proof} By chain rule, 
    \begin{align*}
	\cL \Psi = & \left(\tg' \circ v\right) \cL v+\frac{1}{2} \left(\tg'' \circ v \right) \vert \sigma^\top \nabla v\vert^2\\
	= & \left(\tg' \circ v \right) \cdot 
	\left(\cL v +\frac{1}{2}\frac{\tg''\circ v}{\tg' \circ v}\vert\sigma^\top \nabla v\vert^2 \ldots \right.\\ 
    & \left. + \min\limits_{a\in  \R^k} \left\{-\frac{1}{2}\frac{\tg'\circ v}{\tg'' \circ v} \left|a-\frac{\tg''\circ v}{\tg' \circ v}\sigma^\top \nabla v\right|^2 \right\}		\right)\\
	= & \left(\tg' \circ v \right) \cdot \left(
	\cL v +\min\limits_{a\in  \R^k} \left\{- \frac{\vert a\vert^2 }{2 }\cdot\frac{\tg' \circ v}{\tg'' \circ v}+a\cdot \sigma^\top \nabla v\right\}\right)\\
	= & \left(\tg' \circ v \right) \cdot
	\min\limits_{a\in  \R^k} \left\{ \cL^a v-\frac{\tg' \circ v}{2 \tg'' \circ v}\vert a \vert^2  \right\}.	
	\end{align*}
\end{proof}

For $H\equiv 0$, the PDE in (\ref{linBVP}) transforms according to 
		\begin{align*}
			\cL \Psi - F\Psi &= \left(\tg' \circ v\right) \left(\cL v \right)+\frac{1}{2} \left(\tg'' \circ v \right) \vert \sigma^\top \nabla v\vert^2 -F (\tg \circ v)\\
			&=\left(\tg' \circ v\right) \left(\cL v+\frac{\tg'' \circ v}{2\tg' \circ v}\vert \sigma^\top \nabla v\vert^2  - F~ \frac{\tg \circ v}{\tg' \circ v} \right)\\
			&=\cL v+\frac{\tg'' \circ v}{2\tg' \circ v}\vert \sigma^\top \nabla v\vert^2  -F~ \frac{\tg \circ v}{\tg' \circ v}\,.
		\end{align*}
We mention two relevant special cases that correspond to the control problems of Definitions \ref{SOC1} and \ref{SOC2}:

\begin{corollary}[Logarithmic transformation]\label{cor:logtransform} Let $\tg(x)=e^{-\lambda x}$ and define
\begin{equation}
\Psi(x) = \E^x\!\left[\exp\!\left(-\lambda
\left(\int_0^\tau f(X_s)dt + g(X_\tau)\right)\right)\right]\,.
\end{equation}
Then $\Psi$ is the solution of (\ref{linBVP}) with $H=0$, $F=\lambda f$, and $G=e^{-\lambda g}$ if and only if the log transformed function  $v=\tg^{-1}\circ\Psi$ 
solves the dynamic programming equation (\ref{HJB1}) or, equivalently, (\ref{PDE1}). In other words,
\begin{equation}
v_1(x) = -\lambda\log\Psi(x)\,.
\end{equation}
\end{corollary}
\begin{proof}
By the Feynman-Kac Theorem (cf.~\cite[Thm.~1.3.17]{pham2009continuous} and \cite[Ex.~9.12]{oeksendal2003}, the function $\Psi$ solves the linear boundary value problem (\ref{linBVP}) with $H=0$, $F=\lambda f$, $G=e^{-\lambda g}$. Lemma \ref{lem:transform} with $\tg(x)=e^{-\lambda x}$ implies that 
		\begin{align*}
			\cL \Psi - \lambda f\Psi =\cL v-\frac{\lambda}{2}\vert \sigma^\top \nabla v\vert^2  + f\,,
		\end{align*}
        where the right-hand side equals the first line of the dynamic programming equation (\ref{PDE1}). For $x\in\partial D$, the function $v$ satisfies the boundary condition $v=g$, which shows that (\ref{linBVP}) is indeed equivalent to (\ref{PDE1}). As a consequence,   $-\lambda\log\Psi(x)=\min_u J_1(x,u)$ is the asserted value function.    
\end{proof}

\begin{corollary}[Quadratic transformation]\label{cor:quadtransform}
    Let $\tg(x)=\sqrt{x}$ and define 
   \begin{equation}
\Psi(x) = \E^x\!\left[g(X_\tau)\right]\,.
\end{equation}
for $g\ge 0$. Then $\Psi$ solves (\ref{linBVP}) with $F=H=0$ and $G=g$ if and only if $v=\tg^{-1}\circ\Psi$ is the solution to (\ref{HJB_2}) or, equivalently, (\ref{PDE_2}). Consequently, 
\begin{equation}
v_2(x) = \big(\E^x[g(X_\tau)]\big)^2\,.
\end{equation}
\end{corollary}
\begin{proof}
    The Feynman-Kac Theorem implies that $\Psi$ solves (\ref{linBVP}) with $F=H=0$ and $G=g$, moreover, by Lemma \ref{lem:transform}, 
    \begin{equation}
    \cL\Psi = \cL v-\frac{1}{4v} \left|\sigma^\top \nabla v\right|^2 \,,
    \end{equation}
    where the right-hand side equals the first line of the risk-sensitive HJB equation (\ref{PDE_2}). Together with the boundary condition $v(x)=g^2(x)$ for $x\in\partial D$, this establishes the equivalence between the linear and nonlinear boundary value problems. Since $g\ge 0$, the minimum of the second moment is attained iff the variance is zero, which implies that $\Psi^2(x)=\min_u J_2(x,u)$, as asserted.   	
\end{proof}

\paragraph{Connection to certainty-equivalence}

The transformation $\tg$ plays a similar role as the convex transformation $\varphi$ in Section \ref{sec:zerovar}. 

\begin{itemize}
	\item Clearly, $\tg(x)=e^{-\lambda x}$ is strictly convex and so is its inverse $\tg^{-1}(y)=-\lambda^{-1}\log y$ (because $\tg$ is decreasing). Hence, the choice  $\tg(x)=e^{-\lambda x}$ corresponds to the certainty-equivalence principle of Lemma \ref{Gibbs}, since
\begin{equation}
 V(x) = \tg^{-1}\big(\underbrace{\E^x\!\left[\tg(W(X))\right]}_{\Psi(x)}\big)\,,
\end{equation}
with $W(X)$ as defined by (\ref{pathfunctional}). 
Hence, $\tg$ plays just the role of the convex function $\varphi$ in the certainty-equivalent expectation $\varphi^{-1}(\E[\varphi(W(X))]$.)

\item On the other hand, the function $\tg(x)=\sqrt{x}$ is strictly concave, with strictly convex inverse $\tg^{-1}(y)=y^2$, $y\ge 0$. Using that $\sqrt{g^2}=|g|=g$ for $g\ge 0$, Corollary \ref{cor:quadtransform} can be rephrased as a certainty-equivalence principle for the second moment, with $W(X)=g^2(X_\tau)$ and a strictly concave function:  
\begin{equation}
 V(x) = \tg^{-1}\big(\underbrace{\E^x\!\left[\tg(W(X))\right]}_{\Psi(x)}\big)\,.
\end{equation}
Since 
\begin{equation}
\E^x[g(X_\tau)]=\E^x[\tg(W(X))]\le \tg(\E^x[W(X)])=\sqrt{\E^x[g^2(X_\tau)]}\,,
\end{equation} 
with the rightmost expression (cf.~\cite{daipra1996,anum}) being equal to 
\begin{equation}
\sqrt{\E^x[g^2(X_\tau)]} = \sup_{u\in\cA}\left\{\E^x\!\left[g(X^u_\tau)\,\mathscr{L}_\tau^{-1/2}(X^u)\right]\right\}\,,
\end{equation}
we see that the certainty-equivalence principle involves a \emph{maximization} of $g$ under the controlled dynamics $X^u$ rather than a minimization. Roughly speaking, this means that variance minimization can be achieved by either reducing the second moment or, likewise, by increasing the first moment under the controlled dynamics (which is different from the unbiased importance sampling estimator for the first moment). 

\end{itemize}

To appreciate the difference between the two SOC problems of Definitions \ref{SOC1} and \ref{SOC2}, we consider the quantity of interest 
\begin{equation}
\Psi(x) = P(X_\tau\in C\,|\,X_0=x)
\end{equation}
for some measurable subset $C\subset\partial D$.  We assume that stopping at the target set $C$ is a rare event for most initial values $x\in D$. 
The SOC formulation according to Definition \ref{SOC1} employs $f=0$ and $g=-\log\mathbb{1}_C$, such that  
\begin{equation}
-\log P(X_\tau\in C\,|\,X_0=x) = \min_{u\in\cA}\E^x\!\left[\frac{1}{2}\int_0^\tau |u_s|^2\,ds - \log\mathbb{1}_{C}(X^u_\tau)\right]
\end{equation} 
where for simplicity we have set $\lambda=1$. On the other hand, the SOC problem according to Definition \ref{SOC2} uses $g=\mathbb{1}_C$, which results in
\begin{equation}
P(X_\tau\in C\,|\,X_0=x)^2 = \min_{u\in\cA}\E^x\!\left[\exp\!\left(\int_0^\tau |u_s|^2\,ds + 2\log\mathbb{1}_{C}(Y^u_\tau\right)\right]
\end{equation} 
In both cases, the control $u$ is penalized, but in the first formulation, the control seeks to avoid the non-target set $\partial D\setminus C$, whereas it favors it in the second formulation (with a much stronger penalization of the control though). Both formulations lead to zero variance estimators, but the first formulation results in an importance sampling scheme that increases the likelihood of the rare event, whereas the second formulation does not, despite the sign difference in the controlled dynamics $X^u$ and $Y^u$. (The fact that hitting the target set event becomes even less likely with the second SOC formulation is a simple consequence of the Cauchy-Schwarz inequality.)

\begin{remark}
If all trajectories have a fixed length $T$, the difference between the two formulations may not be strong with regard to their computational complexity. Yet there are cases in which importance sampling estimators that rely on second moment minimization can generate infinitely long trajectories with probability one, despite having zero variance (e.g. when sampling mean first passage times, see \cite{awad2013,anum}). We will revisit this aspect in Section \ref{sec:exit}.  	
\end{remark}

\section{Approximate policy iteration}\label{sec:API}

The efficiency of the importance sampling schemes depends on (1) the  computational cost of sampling and reweighting under the optimal probability measure and (2) the computational cost of computing the optimal change of measure, i.e., the cost of solving the associated optimal control problem. 

In this section, we will devise an iterative scheme for approximating the optimal change of measure by solving the underlying SOC problem. Specifically, we will compute an optimal control by an approximate policy iteration (API) algorithm that, upon convergence, yields an approximation of the optimal control \emph{and} the corresponding value function. As a consequence, API combines the steps (1) and (2) in that it yields a biased approximation of the quantity of interest without an additional reweighting step. While the price to pay is a bias in the estimate of the quantity of interest (because of the nonlinear transformation involved), there is no need to approximate likelihood ratios for possibly very long trajectories that may lead to computational issues, especially when the state space dimension is high (cf.~\cite{agapiou2015importance,suboptimalIS,bengtsson2005curse}).   

The starting point for API is to realize that the solutions to our two stochastic control problems can be written as fixed-point equations on function space. 
For example, for the first formulation that is based on a log transform, the value function can be expressed as a fixed-point equation on the Sobolev space $H^1_0$: 
\begin{equation}
    v_1(x) = \E^{x}\left[\int\limits_{0}^{\tau^*} \left(\frac{\lambda}{2}\vert \sigma^\top\nabla v_1(X_s^*)\vert^2 +f\left(X_s^{*} \right)\right)ds +g\left(X^*_{\tau^*} \right) \right],
\end{equation}
where $X^*$ and $\tau^*$ denote solution and stopping time under the optimal control $u_s^*=-\lambda\sigma^\top\nabla v_1(X_s^*)$.  
We call $J(u):=J_1(-\lambda\sigma^\top u;x)$ and, in a slight abuse of notation, write the fixed point equation associated with (\ref{HJB1}) as 
\begin{equation}
    v_1 = J(\nabla v_1)\,.
\end{equation}
Setting $Q(u):=J_2(\sigma^\top u;x)$, the value function associated with the dynamic programming equation (\ref{HJB2}) satisfies the fixed-point equation 
\begin{equation}
    v_2 = Q\Big(\frac{1}{2}\nabla\log v_2\Big)
\end{equation}

Howard's policy improvement algorithm, that is a variant of API, breaks the fixed point iteration down into two steps, akin to what is done in expectation-maximization algorithms: The first step is a policy evaluation step, in which the cost functional is evaluated for the current control policy. The second step is a policy improvement step, in which the control policy is updated, based on the estimate of the cost functional.   

We formulate API for the two SOC representations of moment generating functions (Section \ref{sec:API1}) and second moments (Section \ref{sec:API2}). 
{The assumptions below guarantee that the underlying boundary value problems admit classical solutions; cf.~\cite[Ch.~VI.5]{fleming2006}. We stress that it is possible to relax several of the assumptions below, but doing this is not the scope of this paper. Yet, as we will demonstrate in Sec.~\ref{sec:numerics} with suitable numerical examples, we observe convergence of the API schemes, even though almost all of the assumptions are violated. 
\begin{assumption}\label{ass:API} We assume that 
    \begin{enumerate}
        \item $D\subset\R^d$ is open and bounded, with a boundary $\partial D$ of class $C^3$
        \item $a=\sigma\sigma^T$ is invertible, with uniformly bounded inverse, i.e. 
        \[
        \sum_{i,j=1}^d a_{ij}(x)\xi_i\xi_j \ge \alpha|\xi|^2\,,\quad x\in\overline{D}
        \]
        for all $\xi=(\xi_1,\ldots,\xi_d)$ and some $\alpha>0$ 
        \item running and terminal costs, $f$ and $g$, are nonnegative and bounded, with $f\in C^1$ and $g\in C^3$ where $f$ has bounded derivatives 
        \item drift and diffusion coefficients in (\ref{sde0}), (\ref{sde-u}) and (\ref{sde+u}) are of class $C^1$.
    \end{enumerate}
\end{assumption}}

\subsection{Moment generating function}\label{sec:API1}

Let $(c_k)_{k\ge 1}$ be a sequence of admissible, stationary feedback control policies, and define the cost value associated with the control policy $c_k=c_k(x)$ to be 
\begin{equation}
    J^k(x) = \E^{x}\left[\int\limits_{0}^{\tau^k} \left(\frac{1}{2\lambda }\vert c_k(X^k_s)\vert^2 +f\left(X_s^{k} \right)\right)ds +g\left(X^k_{\tau^k} \right) \right],
\end{equation}
where we use the notation $X^k$ and $\tau^k$ to denote the process and the stopping time under the control policy $c_k$. Now, by It\^o's formula, 
\begin{equation}\label{linHJB1}
    (\cL^k J^k)(x) + \ell(x,c_k) = 0\,,\quad x\in D
\end{equation}
with the shorthand notation $\cL^k=\cL^{c_k}$ and $\ell(x,c)=\frac{1}{2\lambda}|c|^2+f(x)$. The last equation comes with the boundary condition $J^k(x)=g(x)$ for $x\in\partial D$ that is independent of the control $u^k$. We refer to (\ref{linHJB1}) as the linearized HJB equation. Policy iteration solves the above fixed point equation by iteratively solving the linearized HJB equation (\ref{linHJB1}) and updating the control policy. 

\paragraph*{Policy iteration algorithm} We first state the exact form of the algorithm (i.e., without discretization) that involves iteration of policy evaluation and a policy improvement steps. We comment on discretization issues below. 

\begin{algorithm}[H]\label{algo:PI1}
    \caption{Howard's policy improvement algorithm (log transform case)}
    \begin{algorithmic}
    \Require $c_k$, $k\ge 1$
    \While{$\|J^{k+1}-J^k\|_{L^2}>\mathrm{TOL}$}
    \State Solve $\cL^k J^{k} + \ell(\cdot,c_k) = 0$
    \State Update $c_{k+1}\in\argmin_c \{\cL^c J^{k} + \ell(\cdot,c) \}$
    \EndWhile
    \Ensure $v_1\approx J^{k+1}, c^*\approx -\lambda\sigma^\top\nabla J^{k+1}$
    \end{algorithmic}
\end{algorithm}
Note that $c_{k+1}$ in the policy update step can be explicitly computed, since 
\begin{equation}\label{policyUpdate1}
    c_{k+1}\in \argmin_c \left\{(\sigma c)\cdot \nabla J^{k} + \frac{1}{2\lambda }|c|^2 \right\} = -\lambda\sigma^\top\nabla J^{k}
\end{equation}

\paragraph*{Convergence analysis} The convergence of the algorithm follows essentially from standard techniques, similar to the arguments in \cite{puterman1981convergence,chang1986successive}. The key difference here is that we are dealing with control problems on an indefinite time horizon that involve unbounded, but almost surely finite stopping times. It turns out that the key arguments that are used to prove convergence of the algorithm in the finite time case carry over to our situation.

\begin{lemma}[Monotonicity of the cost]\label{lem:mono1} {Under the regularity conditions of Assumption \ref{ass:API} the linear boundary value problem (\ref{linHJB1}) has a unique classical solution $J^{k}\in C^2(D)\cap C(\overline{D})$ for all $k\ge 1$, and it holds that}
\begin{equation}
J^{k+1} \le J^{k}\,,\quad k\in\N\,.
\end{equation}
\end{lemma}
\begin{proof}
{Theorem 5.1 in \cite{fleming2006} implies that (\ref{linHJB1}) has a unique classical solution $J^k$, with $\nabla J^k$ being bounded in $D$. Therefore the policy update step implies that}  
\begin{equation}\label{policyMono1}
    \cL^k J^{k} + \ell(\cdot,c_k) \ge  \cL^{k+1} J^{k} + \ell(\cdot,c_{k+1})\,.
\end{equation}
Now let $W := J^{k} - J^{k+1}$. Then, adding and subtracting $\ell(\cdot,c_{k+1})$, 
\begin{align*}
    \cL^{k+1}W & = \cL^{k+1}J^k - \cL^{k+1}J^{k+1}\\
    & = \cL^{k+1}J^k + \ell(\cdot,c_{k+1}) - \left(\cL^{k+1}J^{k+1} + \ell(\cdot,c_{k+1})\right)\\
    & \le \cL^k J^{k} + \ell(\cdot,c_k) \\
    & = 0\,,
\end{align*}
where we have used the linearized HJB equation (\ref{linHJB1}) twice and the monotonicity property (\ref{policyMono1}) in going from line 2 to 3. Integrating $W(X_t^{k+1})$ from $t=0$ to $t=\tau$ and taking expectations, using that the resulting stopped local martingale is in fact a martingale since $W\in C^2(D)\cap C(\overline{D})$, it follows that 
\begin{equation}
\E^x[W(X^{k+1}_\tau)] = W(x) + \E^x\!\left[\int_0^\tau (\cL^{k+1}W)(X^{k+1}_s)\,ds\right]\,.
\end{equation}
Since both $J^k$ and $J^{k+1}$ must satisfy the boundary condition $J^\bullet(X^{k+1}_\tau)=g(X^{k+1}_\tau)$, we conclude that $W(X^{k+1}_\tau)=0$. Therefore, the left-hand side is zero, and since $\cL^{k+1}W\le 0$, we can conclude that 
\begin{equation}
W(x)\ge 0\,,\quad x\in D\,.
\end{equation}
Together with $W(x)=0$ for $x\in\partial D$, this implies that $J^k\ge J^{k+1}$ as claimed. 
\end{proof}

Under the additional assumption that the diffusion coefficient $\sigma$ is bounded, the last lemma implies that both cost values converge to the value function and the controls convergence to the unique optimal control.

\begin{theorem}[Convergence of policy iteration]
Under the previous assumptions and $\sup_{x\in D}\|(\sigma(x))\|_F < \infty$, 
Algorithm \ref{algo:PI1} converges. In particular,    
\begin{equation}J^{k} \to v_1\quad\text{and}\quad \nabla J^k\to \nabla v_1\end{equation} 
uniformly on any compact subset of $\overline{D}$. As a consequence, the sequence $c_k$ converges uniformly to the optimal feedback control law $c^*=-\lambda\sigma^\top \nabla v_1$. 
\end{theorem}

\begin{proof}[Proof (sketch)]
We explain only the idea of the proof that is essentially following the lines of \cite[Thm.~3.3]{chang1986successive}. 
By \cite[Thm VI.6.1]{flemingrishel}, both $J^k$ and $\nabla J^k$ have uniform limits $\bar{J}$ and $\nabla\bar{J}$. Moreover, $\nabla^2 J^k$ converges to $\nabla \bar{J}$ weakly in $L^p(D)$ for any $p>1$. This, together with the linearized HJB equation
\begin{equation}
\cL^{k+1}J^{k+1} + \ell(\cdot,c_{k+1}) = 0\,,\quad k\ge 0
\end{equation}
and $\sup_{x\in D}\|(\sigma(x))\|_F\le \infty$, entails 
\begin{equation}\label{weakHJB}
   \cL^{k+1}J^{k} + \ell(\cdot,c_{k+1}) \;\rightharpoonup\; 0\quad \text{weakly in } L^p\,. 
\end{equation}
As a consequence, since
\begin{equation}
\cL^{k+1}J^{k} + \ell(\cdot,c_{k+1}) = \min_c \{\cL^c J^k + \ell(\cdot,c)\}\,,
\end{equation}
the (weak) convergence of $J^k, \nabla J^k$, and $\nabla^2 J^k$ implies that  
\begin{equation}
\cL^{k+1}J^k + \ell(\cdot,c_{k+1})\}\;\rightharpoonup\;\min_c\{\cL^c \bar{J} + \ell(\cdot,c)\}\quad \text{weakly in } L^p\,,
\end{equation}
Since the right hand in the last equation side is zero by (\ref{weakHJB}), we can conclude (by extracting uniformly convergent subsequences all of which have the same limit) that the limit cost value is equal to the value function, $\bar{J}=v_1$, which together with $c^*=-\lambda\sigma^\top\nabla\bar{J}$ being the unique minimizer implies that 
\begin{equation}
c_k=-\lambda\sigma^\top\nabla J^k\;\to\; \bar{c} = -\lambda\sigma^\top\nabla \bar{J} =  -\lambda\sigma^\top\nabla v_1 = c^*\,,
\end{equation}
uniformly on any compact subset of $\overline{D}$.
\end{proof}

\subsection{Second moment minimization}\label{sec:API2}

We now consider the second stochastic control problem, based on a square root transformation. To this end, we assume $f\equiv 0$ and let
\begin{equation}
    Q^k(x) = \E^{x}\left[g^2\left(X^k_{\tau^k} \right)\exp\left(\int\limits_{0}^{\tau^k}\vert c_k(X^k_s)\vert^2\,ds\right) \right],
\end{equation}
As before, we call $\cL^k$ the generator under the control policy $c_k$ that we assume to be uniformly elliptic for every $k\ge 1$. Additional regularity assumptions on the coefficients may be needed for $Q^k$ to solve the elliptic PDE 
\begin{equation}\label{linHJB2}
  (\cL^k Q^k)(x) + (c_k(x))^2 Q^k(x) = 0\,,\quad x\in D\,,
\end{equation}

with boundary data $Q^k(x)=g^2(x)$ for $x\in\partial D$. We will briefly discuss this issue below, for now we assume that (\ref{linHJB2}) has a unique classical solution.

\paragraph*{Policy iteration algorithm} As before, we state the exact form of the algorithm (i.e.,~without discretization) for second moment minimization. We will show below that the algorithm does not unconditionally converge. 

\begin{algorithm}[H]\label{algo:PI2}
    \caption{Howard's policy improvement algorithm (quadratic case)}
    \begin{algorithmic}
    \Require $c_k$, $k\ge 1$, with $\|c_k\|_\infty\le \delta$ sufficiently small
    \While{$\|Q^{k+1}-Q^k\|_{L^2}>\mathrm{TOL}$ and $\|c_{k+1}\|_\infty\le \delta$}
    \State Solve $\cL^k Q^{k} + (c_k)^2Q^k = 0$
    \State Update $c_{k+1}\in\argmin_c \{\cL^c Q^{k} + c^2 Q^k \}$
    \EndWhile
    \Ensure $v_1\approx Q^{k+1}, c^*\approx \sigma^\top\nabla \log Q^{k+1}$
    \end{algorithmic}
\end{algorithm}
Note that $c_{k+1}$ in the policy update step can again be explicitly computed: 
\begin{equation}\label{policyUpdate2}
    c_{k+1}= -\frac{1}{2}\sigma^\top\nabla \log Q^{k} \in \argmin_c \left\{(\sigma c)\cdot \nabla Q^{k} + c^2 Q^k\right\} 
\end{equation}

\paragraph*{Convergence analysis} As before, the policy improvement step implies that 
\begin{equation}\label{policyMono2}
    (\cL^k Q^{k})(x) + (c_k(x))^2 Q^k(x) \ge  (\cL^{k+1} Q^{k})(x) + (c_{k+1}(x))^2Q^k(x)\,
\end{equation}
for all $x\in D$ and all $k\ge 1$. We thus obtain the analog of Lemma \ref{lem:mono1}. 

\begin{lemma}[Conditional monotonicity of the second moment]\label{lem:mono2} {Under the conditions of Assumption \ref{ass:API} equation, we suppose that (\ref{linHJB2}) has a unique classical solution $Q^{k}\in C^2(D)\cap C(\overline{D})$ for all $k\ge 1$.} Then there exists a constant $\delta>0$, such that $\|c_k\|_\infty\le\delta$, $k\in\N$ implies that  
\begin{equation}
Q^{k+1} \le Q^{k}\,,\quad k\in\N\,.
\end{equation}
\end{lemma}
\begin{proof}
Letting $U := Q^{k} - Q^{k+1}$, it follows by adding and subtracting $(c_{k+1})^2Q^k$, 
\begin{align*}
    \cL^{k+1}U & = \cL^{k+1}Q^k - \cL^{k+1}Q^{k+1}\\
    & = \cL^{k+1}Q^k + (c_{k+1})^2Q^k - \left(\cL^{k+1}Q^{k+1} + (c_{k+1})^2Q^k\right)\\
    & \le \cL^{k}Q^k + (c_k)^2Q^k - \left(\cL^{k+1}Q^{k+1} + (c_{k+1})^2Q^k\right)\\
    & = - \left(\cL^{k+1}Q^{k+1} + (c_{k+1})^2Q^k\right)\\
     & = - \left(\cL^{k+1}Q^{k+1} + (c_{k+1})^2Q^{k+1} \right) - (c_{k+1})^2(Q^k - Q^{k+1})\\
     & = -(c_{k+1})^2 U
\end{align*}
where we have used the linearized HJB equation (\ref{linHJB2}) in lines 4 and 6 and the monotonicity property (\ref{policyMono1}) in going from line 2 to 3. This implies that
\begin{equation}
(\cL^{k+1} + (c_{k+1})^2) U \le 0\,,\quad U|_{\partial D} = 0\,,
\end{equation}
where the boundary condition $U|_{\partial D} = 0$ is a consequence of our stopping time definition. Then there exists a constant $\delta>0$, such that $\|c_{k+1}\|_\infty\le \delta$ implies that the operator $-(\cL^{k+1} + (c_{k+1})^2)$ is still nonnegative. (The optimal such constant that preserves nonnegativity is related to the optimal Hardy weight \cite{devyver2014optimal}.) Then the weak maximum principle (e.g.~\cite{pinchover1999maximum}) implies that $-U$ attains it maximum on the boundary $\partial D$, in other words, $U\ge 0$ in $\overline{D}$. 
\end{proof}

The convergence of the Algorithm  \ref{algo:PI2} is stated without proof. Under the conditions that lead to a monotonically decreasing sequence $(Q^k)_{k\ge 1}$, convergence follows by similar arguments as in the log transform case.

\begin{theorem}[Convergence of policy iteration]
    Algorithm \ref{algo:PI2} converges under the assumptions of Lemma \ref{lem:mono2} and $\sup_{x\in D}\|(\sigma(x))\|_F < \infty$, provided that $Q^{k+1}\le Q^k$ for all $k\in\N$. In this case,
    \begin{equation}
        Q^k\to v_2\quad\text{and}\quad \nabla Q^k\to\nabla v_2
    \end{equation}
    uniformly on any compact subset of $\overline{D}$. Moreover, the sequence $c_k$ converges to the optimal feedback control law $c^*\sigma^\top\log v_2$. 
\end{theorem}    

The sufficient conditions for convergence of Algorithm \ref{algo:PI2}  may fail in general, for example, when $Q^k$ attains the value zero at the boundary, since 
\begin{equation}
c_{k+1}\propto\nabla\log Q^k=\nabla Q^k / Q^k\,.
\end{equation} 
We will illustrate numerically that the algorithm converges when the controls remain sufficiently small throughout the iteration and diverges otherwise. For the committor problem, this means that the regularization parameter needs to be chosen sufficiently large.

\subsection{Computational aspects} 
Even though the policy update step in Algorithm \ref{algo:PI1} has an explicit solution, the policy evaluation step requires to numerically solve a potentially high-dimensional PDE and to represent the current iterate of the cost function, $J^k$, in terms of a function basis, hence the name \emph{approximate policy iteration (API)}. 

We avoid a grid-based discretization for solving the elliptic boundary value problem and instead approximate $J^k$ by Monte Carlo from $N$ independent random initial conditions. The function representation of $J^k$, and likewise $u^{k+1}$, is done by projecting the Monte Carlo samples onto radial basis functions (RBFs) $\varphi_1,\ldots,\varphi_L$. Specifically, we use the parametric ansatz 
\begin{equation}
    \hat{J}^k(x) = \sum_{l=1}^L\hat{\theta}_k\varphi_k(x) =:\hat{\theta}^\top\varphi(x)
\end{equation}
where the parameters are determined by linear least squares:
\begin{equation}
 \hat{\theta}\in\argmin_{\theta\in\R^L} \sum_{n=1}^N |J^k(x_i) - \theta^\top \varphi(x_i)|^2\,.
\end{equation}
(We suppress the dependence of the parameters $\hat{\theta}$ on the iteration stage $k$.) The RBF representation and the Monte Carlo approximation introduce numerical errors that may spoil the convergence of the policy iteration scheme. Yet, we show numerically that the numerical discretization does not harm the convergence as long as the RBF basis is rich enough and the number of independent initial conditions is large enough, so API leads to a robust convergent scheme that provides a solution to the underlying stochastic optimal control problem.  

\begin{remark}
    A strength of API is that the linearized HJB equations can be solved by any available numerical methods; see, e.g. \cite{bertsekas2011approximate}. Moreover it is stable under perturbations due to discretization errors \cite{kerimkulov2025mirror}. In most of the relevant applications, the state space is high-dimensional, which excludes the use of grid-based discretizations. In this case, methods of choice are meshless methods, such as PINNs \cite{berrone2023enforcing}, operator learning \cite{peherstorfer2016data}, Koopman-based approaches \cite{kohne2025error}, or Markov state models \cite{schutte2013metastability}, to mention just a few popular examples. 
    
    Moreover, since the PDE that is solved in every iteration step is linear, it can be elegantly combined with projection operator methods (e.g. \cite{zhang2016effective}) that project the coefficients onto a subspace of suitable collective variables as a means to reduce the dimensionality prior to solving the linearized HJB equation; {see also~\cite{hartmann2014optimal,hartmann2018} for a treatment of slow-fast systems using averging or homogenisation, or \cite{amar2025stochastic} for an application of Gyöngy's mimicking marginals approach \cite{gyongy1986mimicking}.} We will discuss this issue in the next section. 
\end{remark}

\section{Committor problem}\label{sec:numerics}

We now consider $X_t\in \R^d$ and two disjoint subsets $A,B\subset\R^d$ with smooth boundaries $\partial A,\partial B\subset\R^d$. Further let $\tau=\min\{T_A,T_B\}$, with 
\begin{equation}
T_A=\inf\{t\geq 0\vert X_t\in A\}\,,\quad  T_B=\inf\{t\geq 0\vert X_t\in B\}
\end{equation}
being the first hitting times of the sets $A,B$. We assume that $T_A$ and $T_B$ are almost surely finite, and we consider the following boundary value problem. 
\begin{definition}[Committor problem]
	Let $D=\R^d\setminus(\overline{A\cup B})$ be an open (not necessarily bounded) set and $\phi\in C^2(D)\cap C(\overline{D})$ be the solution of 
\begin{equation}\label{committoreqn}
	\begin{cases}
		\cL\phi  =0 & \text{in }~ D\\
		\phi =\mathbb{1}_{\partial B} & \text{on }~ \partial D.
	\end{cases}
\end{equation}
We call (\ref{committoreqn}) the committor equation and $\phi$ the (forward) committor function.
\end{definition}
By the Feynman-Kac Theorem, the solution to (\ref{committoreqn}) is given by 
\begin{equation}
\phi(x)=\E^x[\mathbb{1}_{\partial B}(X_{\tau})]\,.
\end{equation} 
for an initial condition $x\in \R^d\setminus(\overline{A\cup B})$. Equivalently, 
\begin{equation}
\phi(x) = P(T_B<T_A|X_0=x)\,,
\end{equation}
is the probability to hit the set $B$ before hitting $A$ when starting at $x$. 
For reversible diffusions (e.g.~with gradient drift $b=-\nabla V$ and constant scalar diffusion coefficient $\sigma=\sqrt{2\beta^{-1}}$) the committor function encodes the relevant information about the ensemble of reactive trajectories from $A$ to $B$, and it is possible to compute, e.g.~equilibrium transition rates from $A$ to $B$, mean first exit times from $A$, or the hitting point distribution on any interface between $A$ and $B$; the same goes for nonreversible dynamics if also the (backward) committor of hitting $A$ before $B$ is known; for details we refer to \cite{vanden2006towards}.   

The following stochastic control representation of the committor function has been studied by several authors, e.g.~\cite{gao2023transition,hartmann2013characterization,anum,yuan2024optimal}.

\begin{definition}[Log transformed committor equation]\label{SOC_Committor}
	Setting $f\equiv 0$, $g=-\log \mathbb{1}_{\partial B}$ and $\lambda=1$ in Definition \ref{SOC1} we have 
	\begin{equation}\label{log_Committor1}
			-\log\phi(x) = \min_{u\in\cA} \E^{x}\!\left[\int_{0}^{\tau^u} \frac{1}{2}\vert u_s\vert^2ds -\log \left(\mathbb{1}_{\partial B} \left(X_{\tau^u} \right)\right) \right],\quad x\in D
	\end{equation}
	where $X^u$ on the right hand side is the solution to the controlled SDE (\ref{sde+u}), and $\tau^u=\min\{T^u_A,T_B^u\}$ is the first hitting time of $A\cup B$ under the controlled dynamics. 
	The value function $v_1=-\log \phi$ solves the HJB equation
	\begin{equation}\label{HJB_Committor1}
		\begin{cases}
			\min\limits_{c\in\R^k}\left \{\cL^c v_1 + \frac{1}{2}\vert c\vert^2 \right\}=0 & \text{in}~D \\
			v_1=-\log\circ\mathbb{1}_{\partial B} & \text{on}~\partial D
		\end{cases}
	\end{equation}
	or, equivalently, 
	\begin{equation}\label{PDE_Committor1}
		\begin{cases}
			\cL v_1 - \frac{1}{2}\left|\sigma^\top  \nabla v_1\right|^2 =0 & \text{in}~ D\\
			v_1=-\log\circ\mathbb{1}_{\partial B} & \text{on}~ \partial D\,.
		\end{cases}
	\end{equation}
\end{definition}

The optimal control, for which the minimum in \ref{log_Committor1} is attained is given by 
\begin{equation}
	u_t^*=c_1^*(X^{u^*}_t)\,,
\end{equation} 
with 
	\begin{equation}
		c_1^*=-\sigma^\top  \nabla v_1=\sigma^\top  \frac{\nabla \phi}{\phi}.
	\end{equation}
Since the forward committor satisfies the boundary conditions $\phi|_{\partial A}=0$ and $\phi|_{\partial B}=1$, it is easy to see, using the Poincar\'e inequality, that the optimal control is repulsive in a neighborhood of $\partial A$. Note that the committor function is nondecreasing as a function of the distance from $\partial A$, at least locally in a small neighborhood of $A$, so $\nabla\phi/\phi$ is pushing away from $A$.

Likewise, the committor equation admits the following representation as a risk sensitive SOC problem according to Definition \ref{SOC2}.

\begin{definition}[Quadratic committor equation]\label{Committor_Quadratic}
	Setting $g=\mathbb{1}_{\partial B}$ in Definition \ref{SOC2}, we have 
	\begin{equation}\label{quad_Committor2}
		\phi^2(x)=\min_{u\in\cA}\E^{x}\!\left[\mathbb{1}_{\partial B}\!\left(Y^u_{\tilde{\tau}^u} \right) \exp\!\left(\int_{0}^{\tilde{\tau}^u}|u_s|^2 ds\right) \right],\quad x\in D\,,
	\end{equation}
	where $Y^u$ is the solution to the controlled SDE (\ref{sde-u}), and $\tilde{\tau}^u$ is the first hitting time of $A$ or $B$ under $Y^u$. 
	The value function $v_2=\phi^2$ solves the HJB equation
\begin{equation}\label{HJB_Committor2}
		\begin{cases}
			\min\limits_{c\in\R^k} \left\{  \cL^c v_2 + |c|^2v \right\}=0 & \text{in}~ D\\
			v_2=\mathbb{1}_{\partial B} &  \text{on}~ \partial D\,,
		\end{cases}
		\end{equation}
		or, equivalently, 
	\begin{equation}\label{PDE_Committor2}
		\begin{cases}
			\cL v_2 - \frac{1}{4v_2} \left|\sigma^\top  \nabla v_2 \right|^2=0 & \text{in}~  D\\
			v_2=\mathbb{1}_{\partial B} &  \text{on}~ \partial D\,,
		\end{cases}
	\end{equation}
\end{definition}

The optimal control for which the minimum in (\ref{quad_Committor2}) is attained is given by
\begin{equation}
	u^*_t = c_2^*(Y^{u^*}_t)
\end{equation}
with 
	\begin{equation}
		c_2^*=\frac12 \sigma^\top  \frac{\nabla v_2}{2v_2}=\sigma^\top \frac{\nabla \phi}{\phi}.
	\end{equation}
Note that the difference between the controlled SDE (\ref{sde+u}) for the log transformed committor and the controlled SDE (\ref{sde-u}) for the square root transformation is a minus sign in front of the control. As a consequence, the optimal control tries to attract the dynamics to the non-target set $A$, in accordance with the discussion at the end of Section \ref{sec:SOC}. We will discuss this aspect in the following example for a simple one-dimensional bistable system.

\begin{example}\label{ex:2wellCommittor}
	We consider a one-dimensional reversible diffusion of the form
	\begin{equation}
	dX_t = -V'(X_t)dt + \sqrt{2\beta^{-1}}dB_t\,,\quad X_0=x\,,
	\end{equation} 
	with a symmetric double-well potential 
	\begin{equation}
	V(x) = \frac{1}{2}(x^2-1)^2\,.
	\end{equation}
	We define $A=(-\infty,1.5)$ and $B=(1.5,\infty)$ and compute the corresponding forward committor $\phi=\phi(x)$ using Algorithm \ref{algo:PI1} where we replace the indicator function $\mathbb{1}_{\partial B}$ by the regularized indicator function $\mathbb{1}_{\partial B}+\epsilon$ to guarantee that the corresponding committor function is strictly positive. 
	
	The optimal control $u^*$ that realizes both value functions $v_1=J_1(\cdot,u^*)$ and $v_2=J_2(\cdot,u^*)$ is a gradient force with the same feedback policy 
	\begin{equation}
	c^*_\epsilon(x) = \sigma^\top\frac{\phi_\epsilon'(x)}{\phi_\epsilon(x)} = \sqrt{2\beta^{-1}}(\log\phi_\epsilon(x))'\,,
	\end{equation}
	where $\phi_\epsilon=\phi+\epsilon$ is the corresponding regularized committor that satisfies the boundary conditions $\phi_\epsilon|_{\partial A}=\epsilon$ and $\phi_\epsilon|_{\partial B}=1+\epsilon$. As a consequence, the optimally controlled SDEs are governed by the biased potentials
	\begin{equation}
		V^\pm_\epsilon(x) = \frac{1}{2}(x^2-1)^2 \pm 2\beta^{-1}\log\phi_\epsilon(x)\,.
	\end{equation}
	In the last equation, the minus sign corresponds to the log transformation case, whereas the {plus} sign is for the square root  transformation.

\begin{figure}
	\centering
	\includegraphics[width=0.65\textwidth]{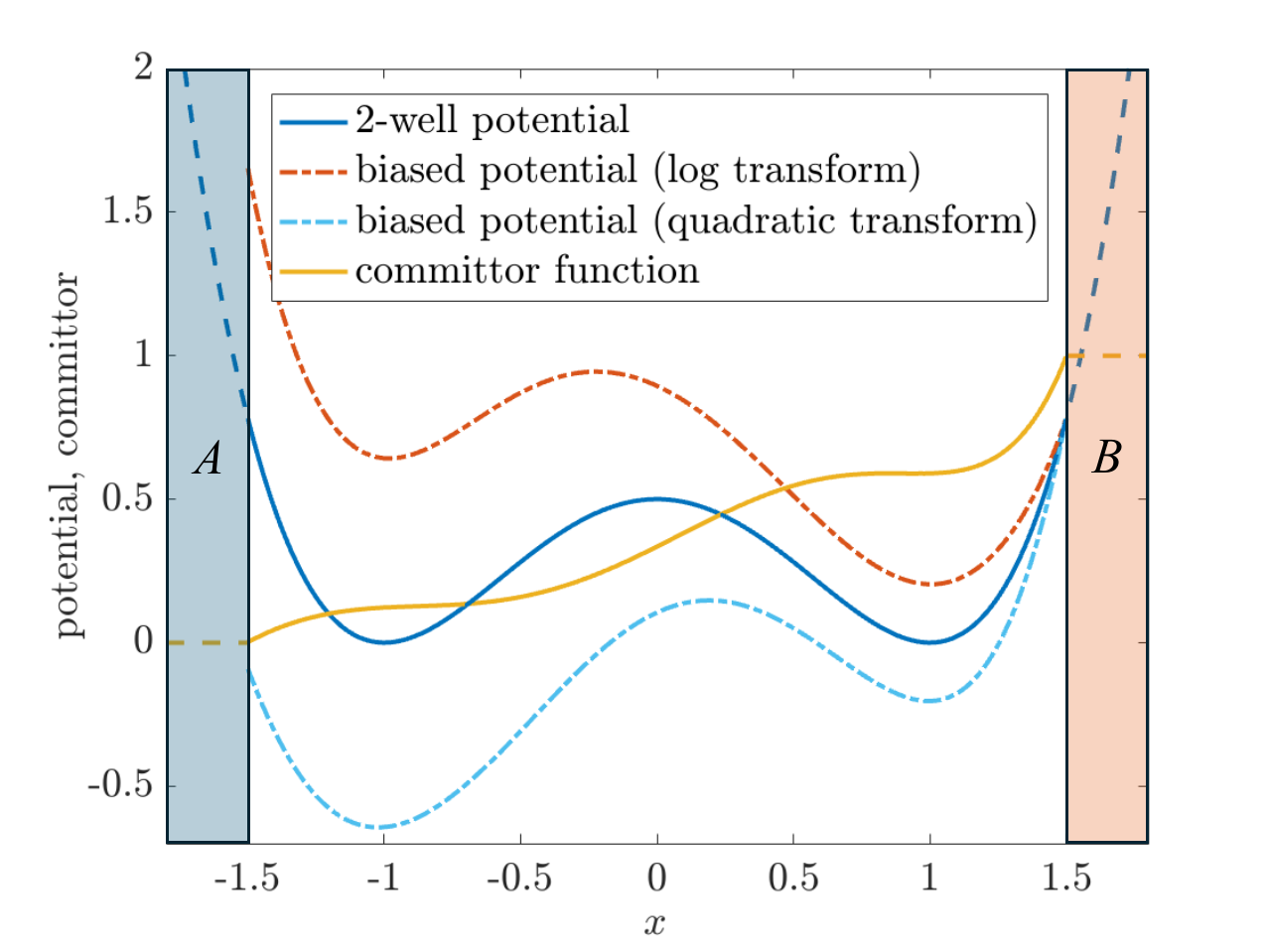}
	\caption{Committor function $\phi$ (orange curve) and the corresponding biased potentials $V^+_\epsilon$ (dashed red curve) and $V^-_\epsilon$ (dashed blue curve) for the symmetric double-well potential (solid blue curve). Unlike the committor, the shape of the bias potential depends on the regularization parameter, $\epsilon$, where here $\epsilon=0.2$.}\label{fig:2wellCommittor}
\end{figure}
	
	Figure \ref{fig:2wellCommittor} shows the double-well potential, the forward committor, and the corresponding biased potentials $V^+_\epsilon$ (dashed red curve) and $V^-_\epsilon$ (dashed blue curve) for inverse temperature $\beta=4.0$ and regularization parameter $\epsilon=0.2$. The plots reveal that, as predicted, the log transformation based representation of the committor decreases the likelihood of hitting the non-target set $A$ whereas the likelihood is increased in case of the square root transformation.  
	
	Note that, since the dynamics is reversible, the backward committor (i.e.~the probability of hitting $A$ before $B$), is given by $1-\phi$, which implies that the control in the case of a square root  transformation increases the likelihood of transitions from $B$ to $A$ in exactly the same way as the log transformation increases the likelihood of going from $A$ to $B$. We stress that the last statement is true due to the fact that our example is one-dimensional. It does not hold in higher dimensions, unless the committor has specific symmetry properties with regard to the sets $A$ and $B$.     

\end{example}

\subsection{A high-dimensional toy example}

We now consider a pure diffusion in $\R^d$ with known committor in order to systematically test and compare the policy iteration algorithms for the two different committor representations. 
To this end, we let 
\begin{equation}
X_t=\sigma B_t+x
\end{equation} 
with $\sigma\in\R\setminus\{0\}$ be a $d$-dimensional Brownian motion starting at $x$ and 
\begin{equation}
	A :=\{x\in\R^d\colon |x|<R_1\}\subset\R^d\,,\quad  B:=\{x\in\R^d\colon |x|>R_2\}\subset\R^d
\end{equation}
open sets. The regularized committor equation reads 
\begin{equation}\label{committoreqnBM}
	\begin{cases}
		\Delta\phi_\epsilon(x)  =0\,, & R_1<|x|<R_2\\
		\phi_\epsilon(x) = \epsilon\,, & |x|=R_1\\
		\phi_\epsilon(x)= 1+\epsilon\,, & |x|=R_2\,.
	\end{cases}
\end{equation}
for some $\epsilon>0$. Note that the committor is independent of $\sigma$, since the prefactor $\frac{\sigma^2}{2}$ in the generator $\cL=\frac{\sigma^2}{2}\Delta$ of $X$ can be dropped. Further note that the parameter $\epsilon$ just leads to a constant shift of the true committor,
\begin{equation}
\phi(x)=P(T_B<T_A|X_0=x)\,,
\end{equation} 
so that the solution to (\ref{committoreqnBM}) is related to $\phi$ by  
\begin{equation}\label{regularization}
\phi_\epsilon(x)= \phi(x)+\epsilon\,.
\end{equation}
The cost functionals associated with the transformed committor are  
	\begin{equation}	J^{\epsilon}_1(x,u)=\E^{x}\left[\int\limits_{0}^{\tau^u} \frac{1}{2}\vert u_s\vert^2ds -\log \left(\mathbb{1}_{\partial B} \left(X_{\tau^u} \right)+\epsilon\right) \right]
\end{equation}
and 
\begin{equation}
	J^{\epsilon}_2(x,u):=\E^{x}\left[ \left(\mathbb{1}_{\partial B}\left(Y^u_{\tilde{\tau}^u} \right)+\epsilon \right)^2 \exp\!\left(\int_{0}^{\tilde{\tau}^u}|u_s|^2 ds\right) \right]\,,
\end{equation}
with the properties
\begin{equation}
-\log(\phi(x)+\epsilon) = \min_{u\in\cA}J^{\epsilon}_1(x,u)
\end{equation}
and 
\begin{equation}
(\phi(x)+\epsilon)^2 = \min_{u\in\cA}J^{\epsilon}_2(x,u)\,.
\end{equation}

\paragraph{Symmetry reduction of the committor equation} 

The domain $D=\R^d\setminus (\overline{A\cup B})$ in (\ref{committoreqnBM}) is rotationally symmetric, and so is the solution. We call $r(x)=|x|$ the distance from the origin, and define a function $\Psi$ by 
\begin{equation}
\Psi(r(x)) := \phi(x)\,.
\end{equation}  
(Likewise, $\Psi_\epsilon(r(x)) := \phi_\epsilon(x)$ defines the regularized radial committor.)
It readily follows that $\Psi$ solves the one-dimensional boundary value problem 
\begin{equation}\label{committoreqnBM2}
	\begin{cases}
		\Delta_r\Psi(r)=0\,, & R_1<r<R_2\\
		\Psi(r)=0\,, & r=R_1\\
		\Psi(r)=1\,, & r=R_2 \,.
	\end{cases}
\end{equation}
where 
\begin{equation}
\Delta_r\Psi(r) = \Psi''(r)+\frac{d-1}{r}\Psi'(r)	 
\end{equation} 
denotes the radial Laplacian. The solution is found by integration: for $d\neq 2$, 
\begin{equation}\label{committorBM}
	\Psi(r)
	=\frac{r^{2-d}-R_1^{2-d}}{R_2^{2-d}-R_1^{2-d}}\,, 
\end{equation}
where the solution for $d=2$ follows from applying l'H\^opital's rule (cf.~\cite{chaosRESIM}), 
\begin{equation}\label{committorBM2}
	\Psi(r)=\frac{\log r - \log R_1}{\log R_2 - \log R_1}.
\end{equation}

The optimal control can now be expressed in terms of the solution to the  radial committor equation (\ref{committoreqnBM2}) by recasting the feedback policy as 
\begin{equation}\label{committorBMctr}
c(x) = \sigma\nabla\log\Psi(r(x)) = \sqrt{2\beta^{-1}}\begin{cases}
	\frac{d}{dx}\log\Psi(x)\,, & d = 1\\
	\frac{x}{|x|}\left.\frac{d}{dr}\log\Psi(r)\right|_{r=|x|}\,, & d>1\,.
\end{cases}
\end{equation}
The left panel of Figure \ref{fig:Committor(r)_2D} shows the symmetry-reduced committor $\Psi=\Psi(r)$ for $d=2$, without regularization, together with the resulting optimal control $u^*=u^*(r)$ as a function of the radius $r$ (in an abuse of notation); the right panel of Figure \ref{fig:Committor(r)_2D} shows the corresponding 2-dimensional vector field (\ref{committorBMctr}). Both plots illustrate that the optimal control in the log transformation representation of the committor  leads to a repulsive force away from the non-target set $A$, pushing in the direction of the target set $B$; for the square root transformation, the control flips sign and the set $A$ becomes absorbing.   

\begin{figure}[H]
	\pgfplotsset{compat=newest,width=0.5\linewidth}
	\begin{tikzpicture}[
		declare function={Psi(\r,\A,\B)=
			ln(\r/ \A)/ln(\B/ \A); 
			u(\r, \A)= 1/(r*ln(r/ \A));
		},]
		\begin{axis}[
			domain=5:10,
			xmin = 5,
			xmax = 10,
			ymin = 0,
			ymax = 5,
			xtick={5,6,7,8,9,10},
			ytick={0,1,2,3,4,5},
			xlabel={$r$\vphantom{$x_1$}},
			ylabel={$\Psi,c^*$}]
			\addplot[smooth, blue]{Psi(x,5,10)};
			\addplot[smooth, red]{1/(x*ln(x/ 5))};
			\legend{$\Psi$, $c^*$}
		\end{axis}
	\end{tikzpicture}
	\pgfplotsset{compat=newest,width=0.5\linewidth}
	\begin{tikzpicture}
		\begin{polaraxis}[  
	xmin=0, xmax=360,
    ymin=5, ymax=10,
	ytick={5,7.5,10},
	xlabel={$r$}
	]
    \addplot3 [
        samples=17, 
        samples y=4,
        quiver={
             u=0, 
			 v=1/(y/5*ln((y+0.04)/5)),
			 scale arrows=0.02
        },
        -stealth,
		data cs=polar,
        domain=0:360,
        domain y=5:10] {1};
  \end{polaraxis}
	
\end{tikzpicture}	
\caption{{The left panel shows the forward committor (\ref{committorBM2}) for $R_1=5$ and $R_2=10$ and the resulting optimal feedback $c^*(r)$ as a function of the radius. The right panel shows the 2-dimensional polar coordinate representation of $c^*$. 
}}\label{fig:Committor(r)_2D}
\end{figure}

\begin{tikzpicture}
  
\end{tikzpicture}

\subsection{Approximate policy iteration}

We now compare the analytical results with numerical approximations obtained by approximate policy iteration (API) using the both MGF representations and the one based on the second moment of the committor. Recall that only in the former case we expect the API algorithm to be unconditionally convergent, whereas in the latter case we expect the algorithm to require strong regularization to prevent too large control values in the neighborhood of the set $A$. 

As for the parameter-linear representation of the cost-value we consider three different types of strictly positive RBF $\varphi_l\in C^\infty$, specifically, we consider 
\begin{enumerate}
    \item[(a)] Gaussian RBFs
    \begin{equation}
        \varphi_l(r) = e^{-\varepsilon^2(r-\mu_l)^2}
    \end{equation}
     \item[(b)] inverse quadratic RBFs
    \begin{equation}
        \varphi_l(r) = \frac{1}{1 + \varepsilon^2(r-\mu_l)^2}
    \end{equation}
     \item[(c)] inverse multiquadric RBFs
    \begin{equation}
        \varphi_l(r) = \frac{1}{\sqrt{1 + \varepsilon^2(r-\mu_l)^2}}
    \end{equation}
\end{enumerate}
for some fixed kernel-width $\varepsilon>0$ and suitable $\mu_l\in\R$, $l=1,\ldots,L$; see \cite{buhmann2000radial} for a discussion of different choices of RBFs and their parameters. 

In the example at hand, the committor is radially symmetric and we therefore represent the committor approximation by 
\begin{equation}
\Psi_\theta(x) = \theta^\top\varphi(|x|)\,,
\end{equation}
where we use the shorthands $\varphi=(\varphi_1,\ldots,\varphi_L)^\top$ and $\theta=(\theta_1,\ldots,\theta_L)^\top\in\R^L$. We further denote by $\widehat{\Psi}=\Psi_{\hat\theta}$ the least squares approximation. That is, given noisy samples $(x_i,y_i)$, $i=1,\ldots,N$ of initial data and committor values, we compute a committor approximation in terms of RBF by minimizing the empirical risk
\begin{equation}
R(\theta) = \frac{1}{N}\sum_{i=1}^N |y_i - \theta^\top\varphi(|x_i|)|^2\,.
\end{equation}
Here the design points $x_1,\ldots,x_n$ are independent draws from some appropriate probability distribution, typically the Boltzmann distribution associated with the underlying dynamics, but other choices are possible too. 

\begin{figure}
    \centering
    \includegraphics[width=0.65\linewidth]{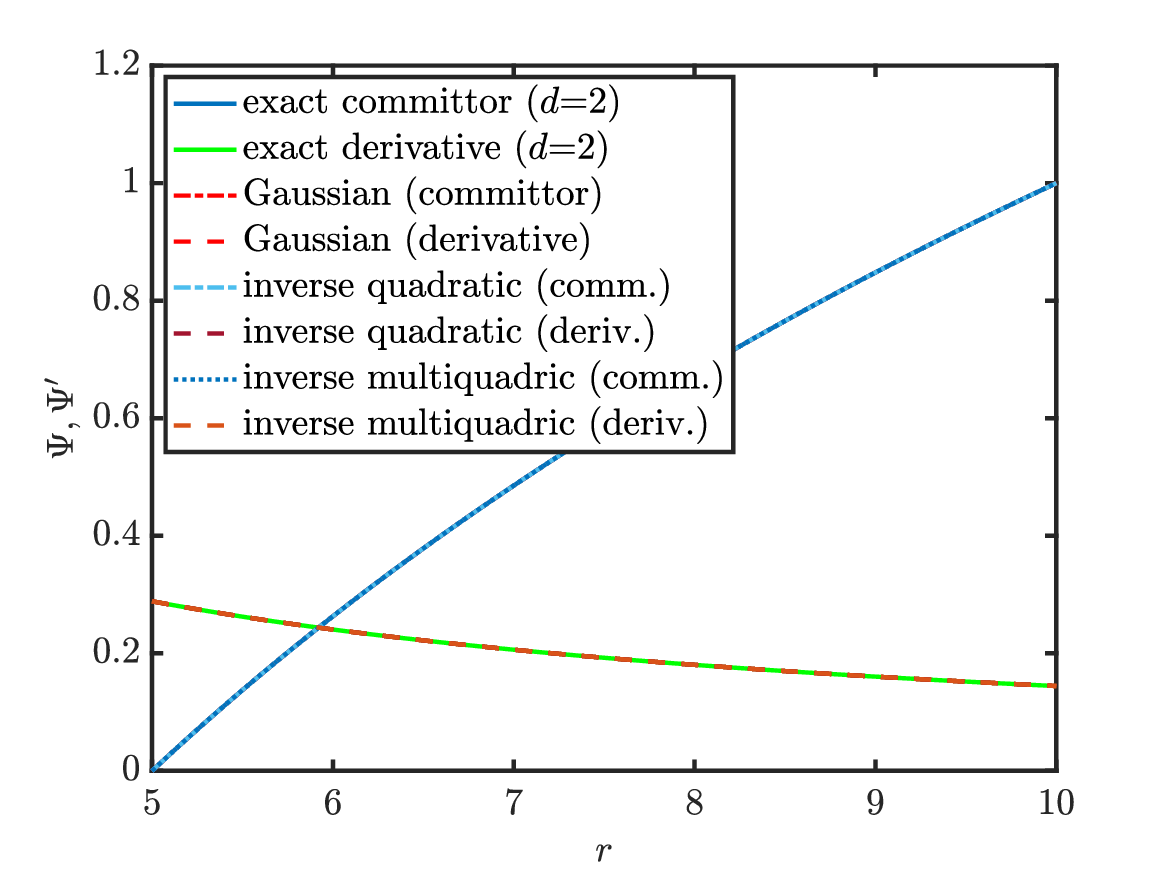}
    \includegraphics[width=0.65\linewidth]{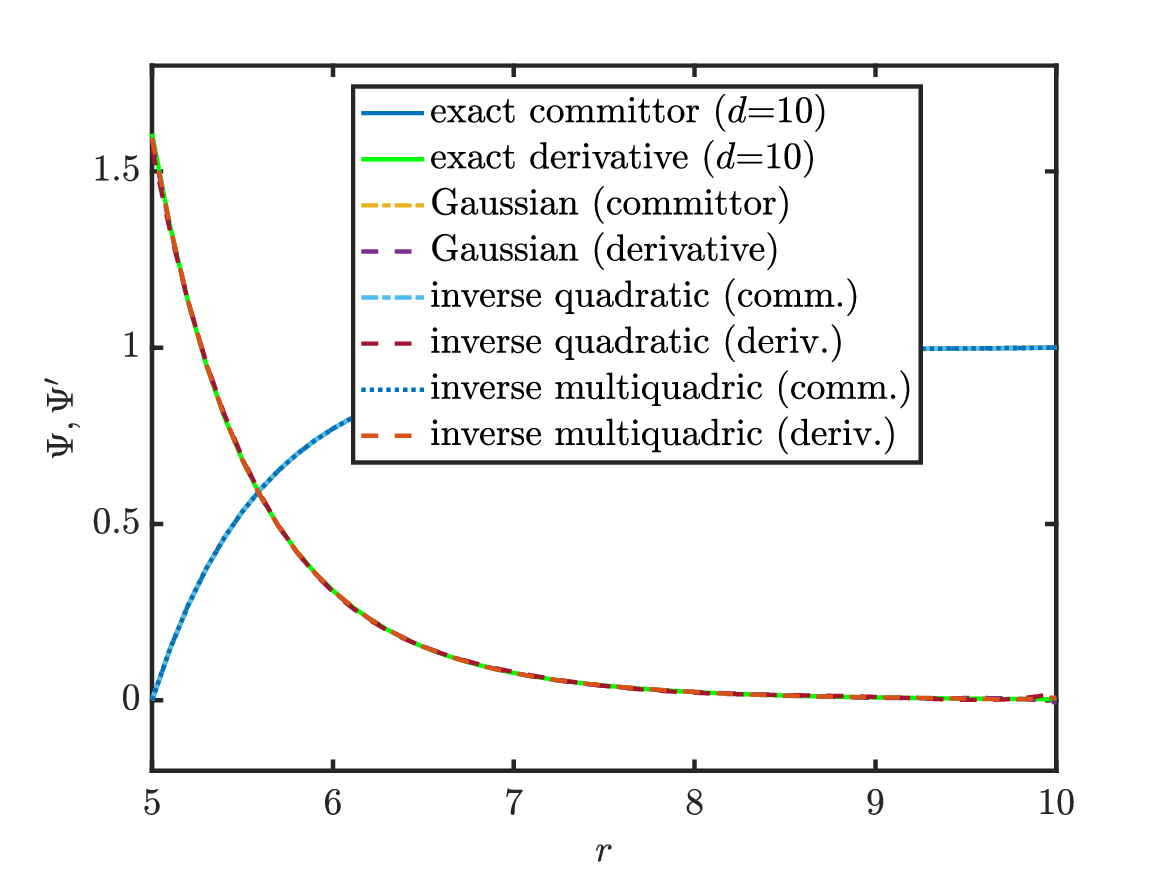}
    \caption{RBF approximation of the radial committor functions and  their derivatives for $d=2$ (upper panel) and $d=10$ (lower panel).}
    \label{fig:RBFapprox}
\end{figure}

Figure \ref{fig:RBFapprox} shows the linear least squares approximations of the committors and their derivatives in dimension $d=2$ (upper panel) and $d=10$ (lower panel) for all three RBFs, with $\varepsilon=0.25$ (Gauss), $\varepsilon=0.05$ (inverse quadratic) and $\varepsilon=0.1$ (inverse multiquadric). The $L=11$ RBF centers  have been set to $\mu_1=5.0,\,\mu_2=5.5,\,\ldots,\, \mu_{L}=10.0$, and we have sampled the committor using $N=51$ equidistant points from the interval $[5,10]$. For this admittedly simple toy example, we observe that the approximation is quite robust with respect to the choice of the kernel width, $\varepsilon$, the number of basis function, $L$, and the RBF centers, $\mu_l$ across different spatial dimensions $d$. 

In the following, we confine ourselves to Gaussian RBFs that allow for the best trade-off between accurate approximation and fast function evaluation.  

\paragraph*{Moment generating function}

\begin{figure}
    \centering
    \includegraphics[width=0.65\linewidth]{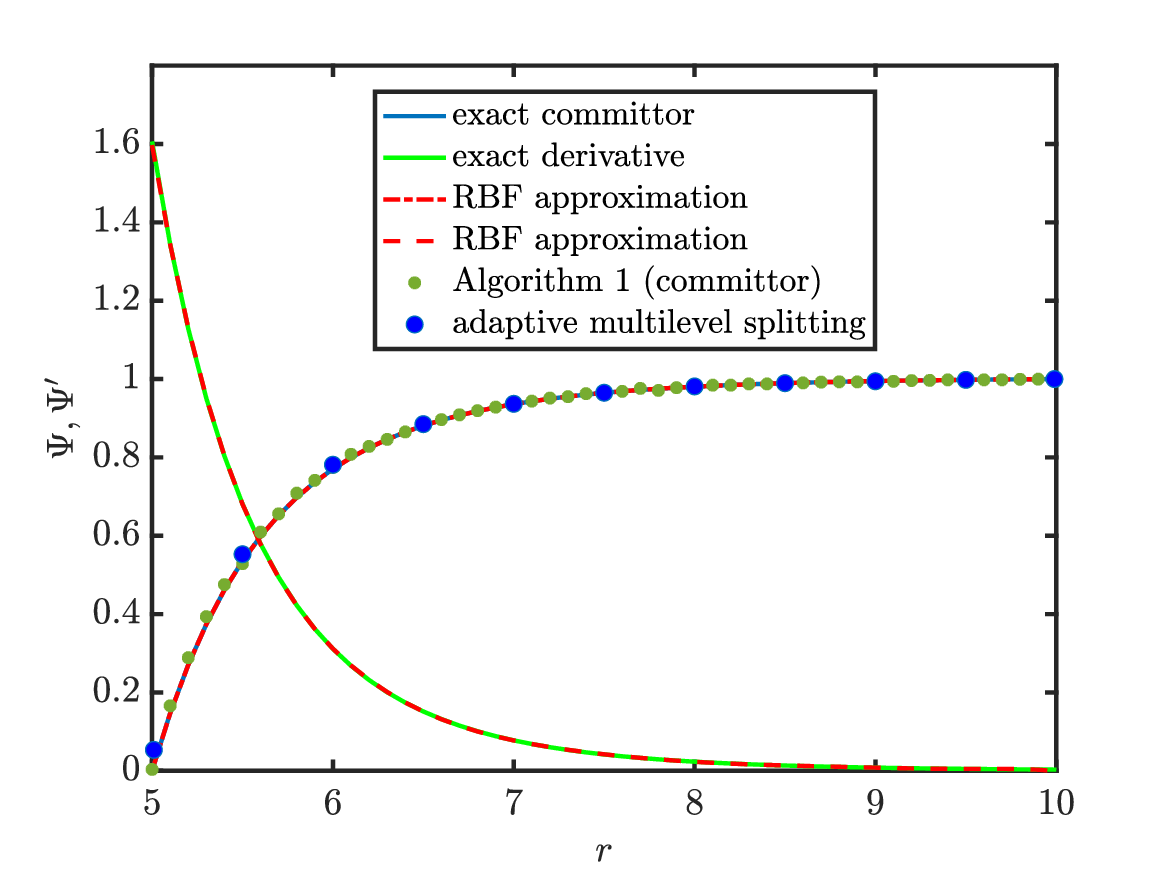}\\\includegraphics[width=0.65\linewidth]{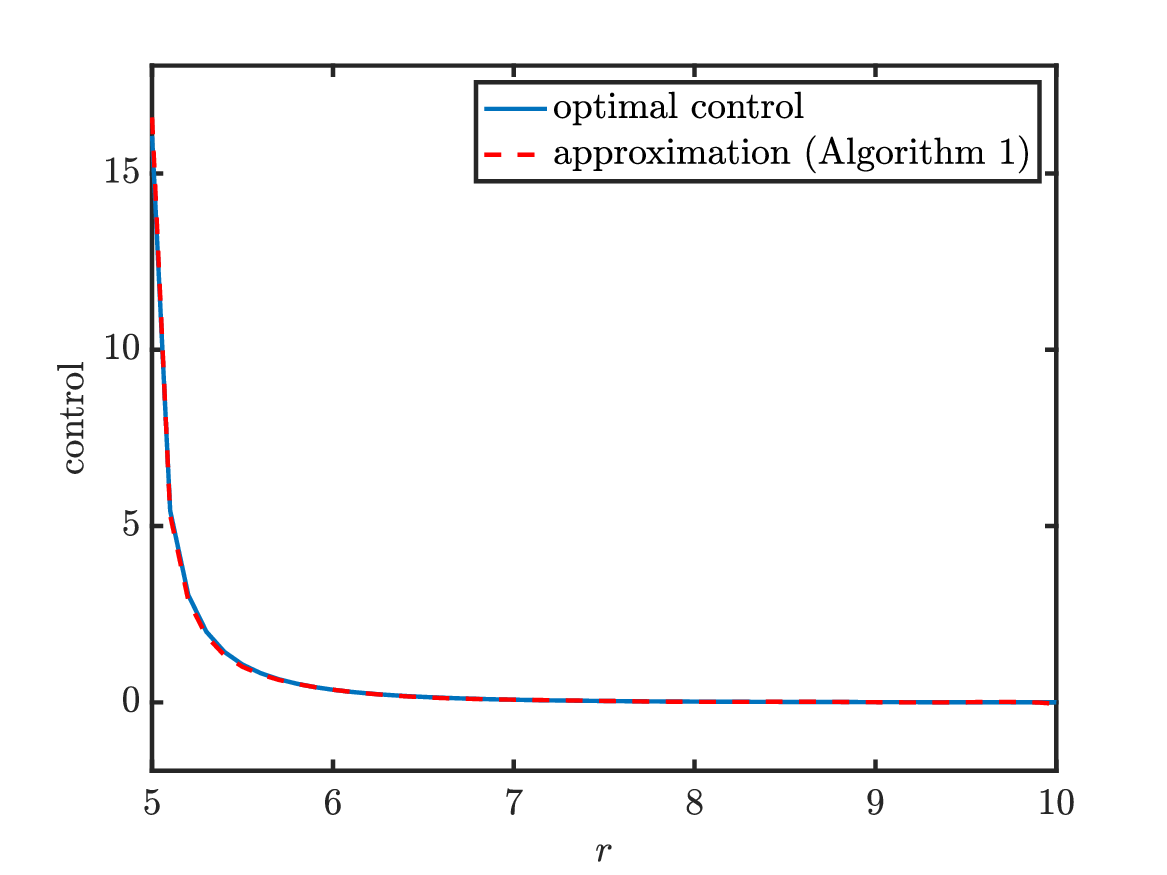}
    \caption{API approximation of the committor based on the MGF representation of the committor. {For comparison, we have included the result of the unbiased adaptive multilevel splitting (AMS) scheme of \cite{ams}}.}
    \label{fig:API1}
\end{figure}

We consider the committor for $d=10$, and compute its MGF-based approximation using Algorithm \ref{algo:PI1}, with a Monte Carlo approximation on the cost value. Specifically, we discretize the radius $r\in[5,10]$ using a regular grid with spacing $\Delta r=0.1$ and launch $N\in\{10^2,10^3,10^4\}$ trajectories from independently and uniformly drawn initial values for each value of $r$. The Euler-Maruyama step size has been set to $\Delta t=0.005$, and the API has been initialized with policy
\begin{equation}
c^1(x) = -\sigma^\top\theta_0^\top\nabla\varphi(|x|)
\end{equation}
for a random parameter $\theta_0$ with standard Gaussian distribution, where as before we have set $\lambda=1$. 

We find that the API algorithm needs between $2$ (for $N=10^4$) and $11$ steps (for $N=10^2$) to converge where we have set the stopping criterion of the algorithm to $|J^{k+1}(r)-J^k(r)|\le 10^{-1}$, with the error measured in the Euclidean norm over 51 test points $r\in[5,10]$. The regularization parameter $\epsilon$ has been set to $\epsilon=0.1$ (cf.~equation \ref{regularization}), but we observe convergence of the API algorithm for a fairly large range of regularization parameters $\epsilon$. Figure \ref{fig:API1} shows the committor (upper panel) and the last iterate of the control policy  (lower panel) for $N=10^4$ and regularization parameter $\epsilon=0.1$.  
For $N>10^3$ the cost values show the theoretically predicted monotonic behavior (for most test points), in accordance with Lemma \ref{lem:mono1}. For smaller sample size $N$, the cost value initially decrease, but then fluctuate around small values. 

{For a rough comparison, Figure \ref{fig:API1} shows the result of a committor computation using an unbiased minimum adaptive multilevel splitting (AMS) scheme of \cite{ams} for the same sample size, $N=10^4$, with uniform initial data on spheres of constant radius, and an importance function given by the radius; the parameters have been chosen such that the computational effort of AMS is roughly the same as for API. Choosing constant $N$ for all initial radii is clearly not optimal in terms of computational efficieny, and  both API and AMS would benefit from an adaptive choice of $N$, nevertheless the simulations show good (visual) agreement of the resulting committor values. We refrain from delving into the details of multilevel algorithms and instead refer to the relevant literature \cite{ben2025multilevel,wagner2020multilevel}.}


\paragraph*{Second moment minimization}

\begin{figure}
    \centering
    \includegraphics[width=0.65\linewidth]{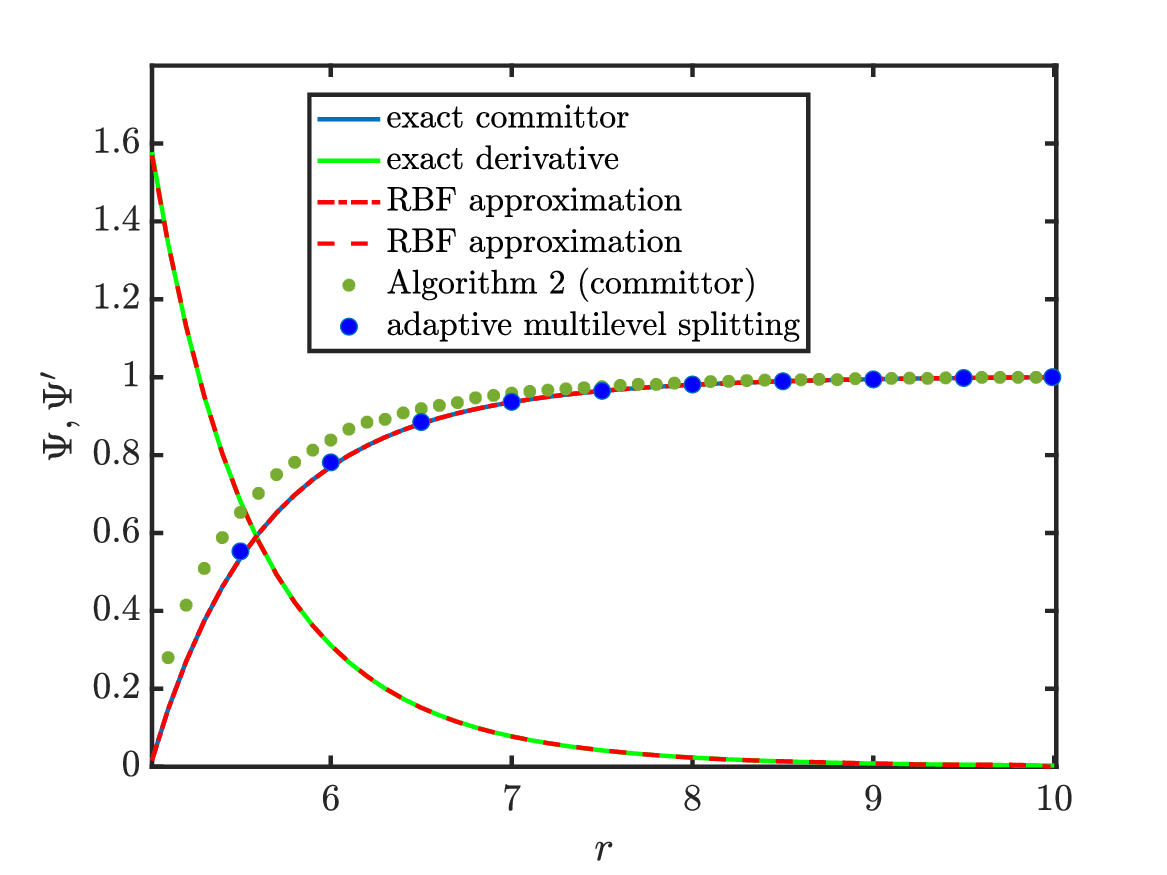}\\
    \includegraphics[width=0.65\linewidth]{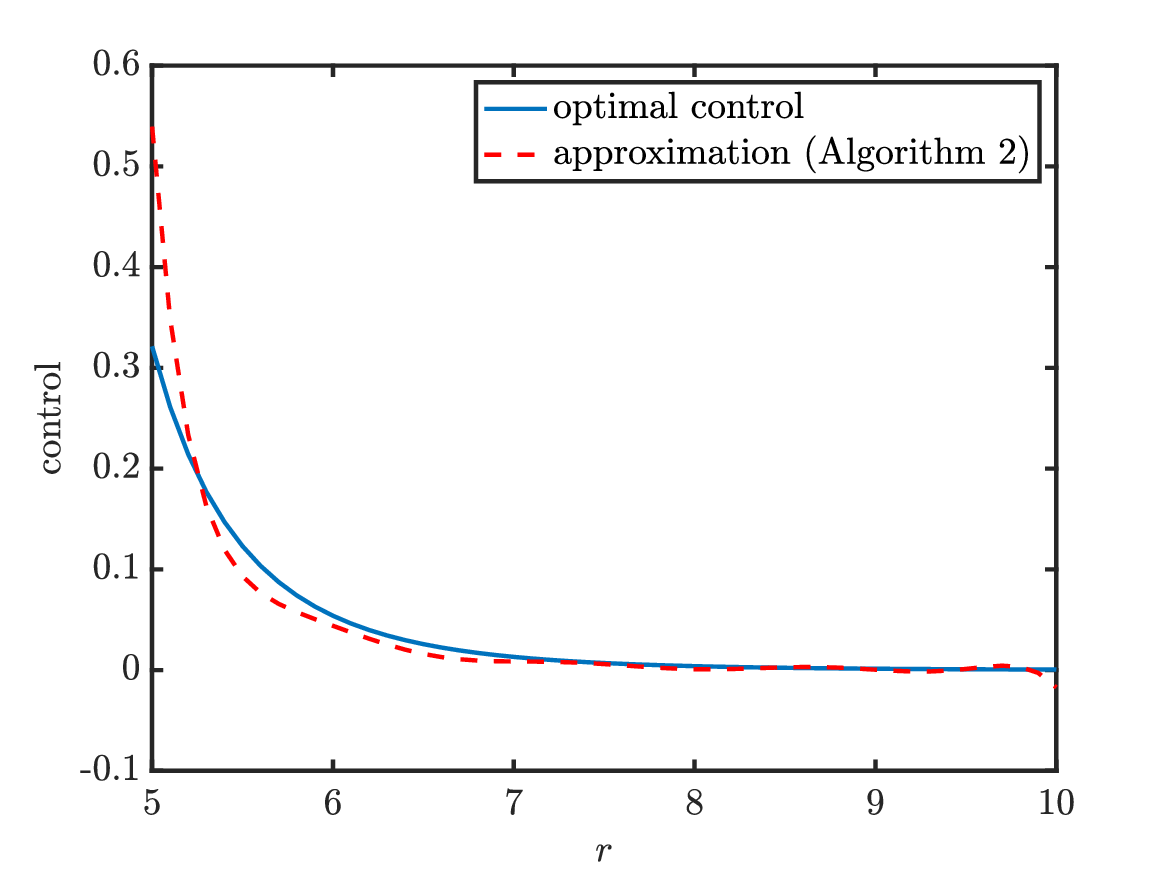}
    \caption{API approximations of the committor based on the second moment representation of the committor. {For comparison, the blue dots show the result of the (unbiased) AMS approximation.}}
    \label{fig:API2}
\end{figure}

We repeat the simulation, but now use the minimum second moment representation of the committor function. {(For comparison, the result of minimal AMS is shown too.)} The simulation parameters are the same as before, except the regularization parameter $\epsilon$. Figure \ref{fig:API2} shows the committor approximation (upper panel) and the resulting control policy (lower panel) for a regularization parameter $\epsilon=5$. 

It can be seen that the resulting approximation of the committor is biased, which is a consequence of the concavity of the square root function: If convergence is reached, then Algorithm \ref{algo:PI2} yields an approximation $\widehat{Q}=\frac{1}{N}\sum_{i=1}\widehat{Q}_{i}$ of the squared committor function $\widehat{\Psi}^2$ since, by Jensen's inequality, 
\begin{equation}\label{biasedEstimateQ}
\widehat{\Psi}=\sqrt{\widehat{Q}}\quad \Longrightarrow \quad \sqrt{\frac{1}{N}\sum_{i=1}^N \widehat{Q}_i} \ge \frac{1}{N}\sum_{i=1}^N \sqrt{\widehat{Q}_i} = \frac{1}{N}\sum_{i=1}^N \widehat{\Psi}_i =\widehat{\Psi}\,,
\end{equation}
provided that $\widehat{Q}_i$ is an unbiased estimator of the minimum second moment $\widehat{\Psi}_i^2$ under the optimal control. This explains the visible overestimation of the small committor probabilities by the API approximation, when using Algorithm \ref{algo:PI2}. Note that the MGF based estimator is biased too, since the value function equals the log committor where the log function is concave too. Yet the estimator variance is much closer to zero, which reduces the bias of the estimator.  

We shall briefly comment on the influence of the regularization parameter. For values of $\epsilon$ between 0.5 and 5, API still converges, but relatively slowly. This behavior is in accordance with Lemma \ref{lem:mono2}: for too small values of $\epsilon$, Algorithm \ref{algo:PI2} diverges as the controls become too large, and it converges when the regularization parameter is large enough, thus forcing the control to stay sufficiently small in a neighborhood of $R_1$. Numerically, this is illustrated in Figure \ref{fig:API-eps} that shows the relative cost value $|Q^{k+1}-Q^k|$ in the Euclidean norm where, for comparison, the relative cost values have been normalized so that the first value is always equal to one. (The actual cost values depend on $\epsilon$.)  

{\begin{remark}
    In practice, the regularization has the effect of capping the control close to the boundary of the non-target set. The optimal regularization parameter (that is just large enough to enforce monotonicity of the API algorithm for the second moment) is related to the optimal Hardy weight for the generator $\cL$, which depends on its spectral gap; roughly speaking, a smaller spectral gap requires stronger regularisation and hence stronger force capping. As a consequence, the convergence rate of the algorithm degrades if the neighbourhood of the non-target set is highly metastable. In this case, using the second moment equivalent of the committor equation is not recommended, and API based on the cumulant generating function is likely to be more efficient.  
\end{remark}}

\begin{figure}
    \centering
    \includegraphics[width=0.65\linewidth]{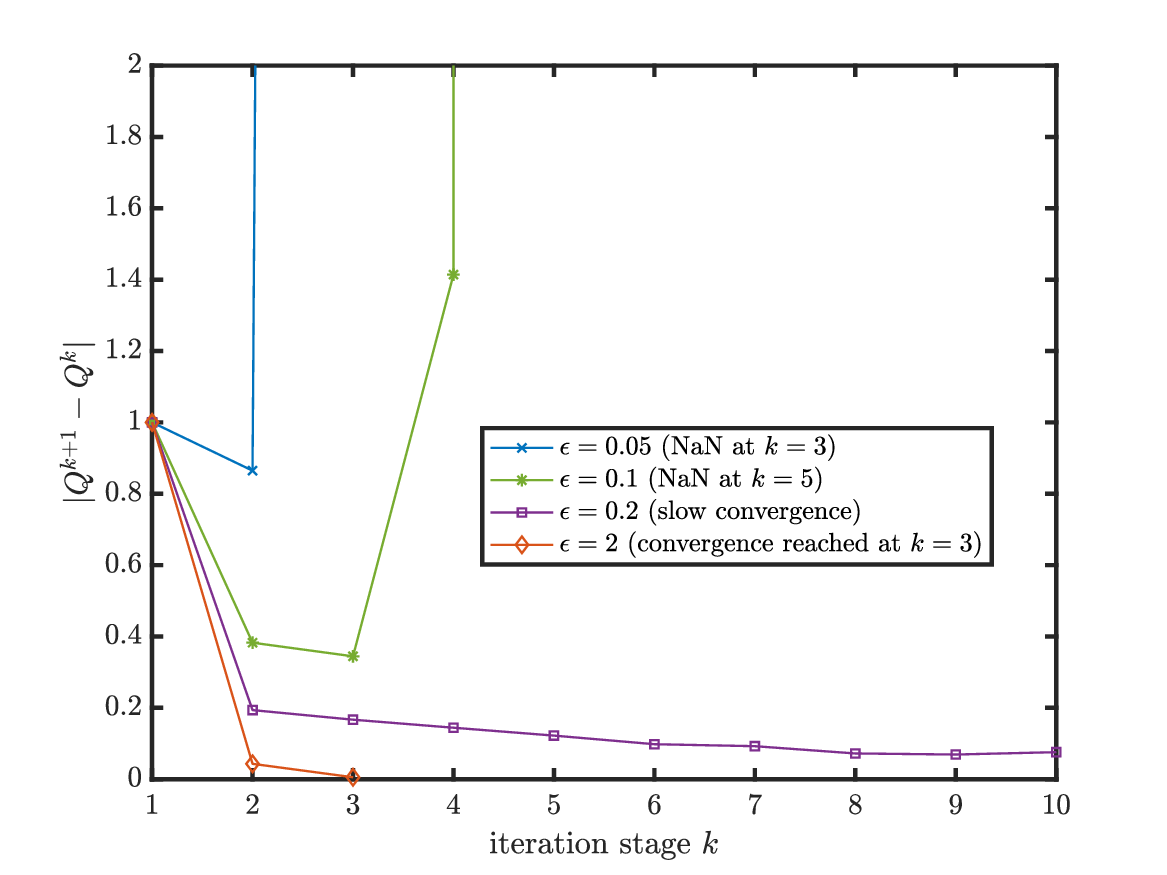}
    \caption{Influence of the regularization parameter $\epsilon$ on the convergence or divergence of Algorithm \ref{algo:PI2}. If the regularization is too weak then the operator in the linearized HJB equation ceases to be negative semidefinite, which leads to an exponential growth and divergence of the cost values.}
    \label{fig:API-eps}
\end{figure}

\section{Discussion}\label{sec:discussion}

{In this section, we sketch the necessary steps that need to be taken in order to apply the API algorithm to molecular dynamics, in particular when using phase space dynamics (Section \ref{sec:adp}). We moreover discuss the application of the log transformation idea to the notoriously difficult exit time problem (Section \ref{sec:exit}).}  

{
\subsection{Application to molecular dynamics on phase space}\label{sec:adp}}

{
We treat a typical benchmark problem from molecular dynamics: the conformational C7$_{\rm eq}$-C5 transition of alanine dipeptide (ADP). To this end we consider an underdamped Langevin simulation of ADP in vacuum at $T=300K$ with moderate to large friction. For the sake of convenience, we confine our attention to the case of the log transformed committor; we will briefly comment of the square root transformation below.}

{
One of the challenges is that, while the committor function for the  C7$_{\rm eq}$-C5 transition is typically represented as a function of the two central torsion angles, $\phi,\psi$, the optimal feedback policy is a function of other variables too. In particular, for underdamped Langevin dynamics, the committor is a function of the momenta, and optimal control is the gradient of the log committor  \wrt the momenta. 
To appreciate the last statement, we recast the uncontrolled underdamped Langevin equation for a system of particles with positions $r\in \R^{n}$ and momenta $s\in\R^n$ in the following form
\begin{equation}\label{uld}
\begin{aligned}
	dR_t & = M^{-1}S_tdt\\
	dS_t & = -\nabla V(R_t) dt - \gamma M^{-1}S_t dt + \sqrt{2\beta^{-1}\gamma}\, dW_t \,,
\end{aligned}
\end{equation}
where $\beta=(k_B T)^{-1}$ is a positive constant, $M\,,\gamma\in\R^{n\times n}$ are the symmetric and positive definite mass and friction matrices, and $V\colon\R^n\to \R$ is the smooth molecular interaction potential. Equation (\ref{uld}) is of the form (\ref{sde0}), with 
\[
b(x) = \begin{pmatrix}M^{-1}s \\ -\nabla V(r) - \gamma M^{-1}s\end{pmatrix}\in\R^{2n}\,,\; \sigma(x) = \begin{pmatrix}0\\ \sqrt{2\beta^{-1}\gamma}\end{pmatrix}\in\R^{2n\times n}\,,
\]
for the state vector $x=(r,s)\in\R^{2n}$ and $\sqrt{\gamma}\in\R^{n\times n}$ denoting the unique symmetric positive definite matrix square root of $\gamma\in\R^{n\times n}$. Note that it does not satisfy the uniform ellipticity requirement of Assumption \ref{ass:API} since the noise coefficient is degenerate. (The generator of \ref{uld} is hypoelliptic, which is why we nevertheless expect no issues with the API scheme.) }

{
\begin{figure}
	\centering
	\includegraphics[width=0.5\textwidth]{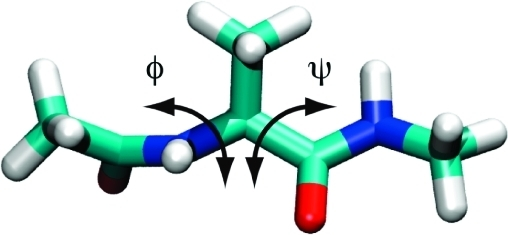}
\caption{Ball-and-stick model of alanine dipeptide (ADP) highlighting the backbone torsion angles $\phi$ and $\psi$. Image adapted from ETH Zürich CP2K Materials Modeling course website (see \url{https://www.cp2k.org/exercises:2015_ethz_mmm:alanine_dipeptide})}\label{fig:adp}
\end{figure}}

{For ADP, the target and non-target sets $A$ and $B$ that define the committor function $Q(x):=P(T_A>T_B|X_0=x)$ depend only on the configuration variables $r\in\R^n$. 
Yet, the committor is a function of the full state vector $x\in\R^{2n}$. 
Moreover, by Definition \ref{SOC1} and Lemma \ref{lem:socGibbs}, the optimal control that realizes the zero-variance estimator of the committor function involves a gradient of the log committor with respect to the momentum variables, since   
\begin{equation*}
 u^*_t = \sigma^\top \nabla\log Q(X_t^*) =  \sqrt{2\beta^{-1}\gamma}\,\nabla_s\log Q(X_t^*)\,,\quad X^*_t=(R^*_t,S^*_t)\,.
\end{equation*}
While this representation is a direct consequence of  Lemma \ref{lem:socGibbs}, it is at odds with the usual practice in molecular dynamics that is based on  representation of committor functions using (few collective) configuration variables.}

{\paragraph*{High friction approximation of the optimal control} We are not aware of any algorithm that is designed to use committor representations in terms of momenta, but as we show in Appendix \ref{app:hfLimit}, using a formal asymptotic expansion of the committor equation, the optimal control $u^*$ in the high friction regime can be approximated by the feedback control 
\begin{equation}\label{hfControl}
	u^*_t \approx \sqrt{2\beta^{-1}\gamma}\,\nabla\log q(R^u_t)
\end{equation}
where $q$ is the configuration space committor associated with the overdamped Langevin equation. Details are given in Appendix \ref{app:hfLimit}; cf.~also \cite{hartmann2014optimal}.}

{This now suggests to build the API scheme on the fixed point iteration 
\begin{equation}\label{adpAPI}
	-\log q_{k+1}(r) = \bar{J}(\sqrt{2\beta^{-1}\gamma}\,\nabla\log q_{k};r)
\end{equation}
with 
\begin{equation}\label{adpJ}
	\bar{J}(u;r) = \E\left[\int\limits_{0}^{\tau^u} \frac{1}{2}\vert u_s\vert^2ds -\log \left(\mathbb{1}_{\partial B} \left(R^u_{\tau^u} \right)+\epsilon\right) \,\Big|\, R^u_0=r,\,S^u_0\sim\mu_0\right]
\end{equation}
for the initial momentum distribution $\mu_0=\cN(0,\beta^{-1}M)$, some  regularisation parameter $\epsilon>0$, and $R^u$ the solution to the controlled underdamped Langevin equation
\begin{equation}\label{uld+u}
\begin{aligned}
	dR^u_t & = M^{-1}S^u_tdt\\
	dS^u_t & = -\nabla V(R^u_t) dr - \gamma M^{-1}S^u_t dt + \sqrt{2\beta^{-1}\gamma}\,u_t\,dt + \sqrt{2\beta^{-1}\gamma}\, dW_t \,.
\end{aligned}
\end{equation}
}

{\paragraph*{Committor approximation for ADP} Figure \ref{fig:committorADP} shows the resulting committor function for the system (\ref{adpAPI})--(\ref{uld+u}) for different values of the regularisation parameter $\epsilon$ and as a function of the two central backbone dihedral angles. The simulations were carried out using the Julia implementation of ISOKANN \cite{isokann} for a friction coefficient $\gamma=10^3 I_{n\times n}$, using the BAOAB integrator with step size $\Delta t= 10^{-5}$. In every case, convergence was reached after 4-6 policy iterations.
During the API cycles, the log committor functions have been interpolated with  cubic splines in the $(\phi,\psi)$ Ramachandran plane, using a $20\times 20$ grid with periodic boundary conditions imposed on the spline functions. The representation of the committor functions -- and hence the optimal control approximants -- in terms of the two central backbone torsion angles is justified by the fact that the two angles are the dominant slow variables for ADP in vacuum.}


{While the results were relatively consistent throughout different choices of the simulation parameters (e.g. step size, moderate to large friction coefficient, maximum number of integration steps such that at least 90\% of the realisations terminate in $A$ or $B$, etc.), we observed a stark influence of the regularisation parameter $\epsilon$ on the committor approximation. Clearly, as the previous toy examples demonstrates, the speed of convergence of the API can degrade when the force capping is strong due to large regularisation. This explains why for large $\epsilon$ (see botton right panel of Figure \ref{fig:committorADP}) convergence is considerably slower than for small $\epsilon$. For ADP, however, we cannot observe any significant changes in the committor values after 4-6 iterations, even for large values of $\epsilon$. (For small regularisation, say $\epsilon\le 0.05$ convergence was reached after only 2 API cycles.) We conjecture that the MSM based committor approximations are still relatively rough, which leads to large gradients in the log committor approximation so that the force capping effect for large values of $\epsilon$ essentially discards the control.}

{
\begin{figure}
	\centering
	\includegraphics[width=0.495\textwidth]{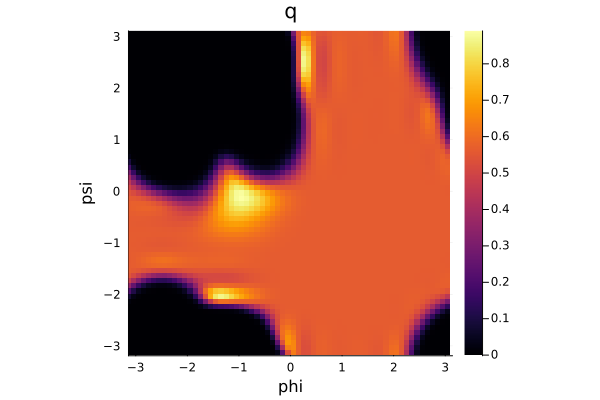}
	\includegraphics[width=0.495\textwidth]{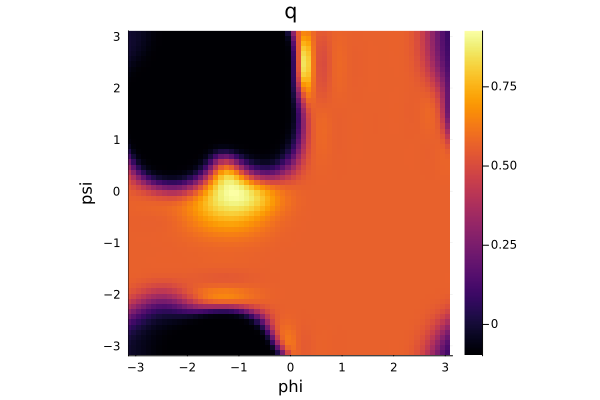}\\
	\includegraphics[width=0.495\textwidth]{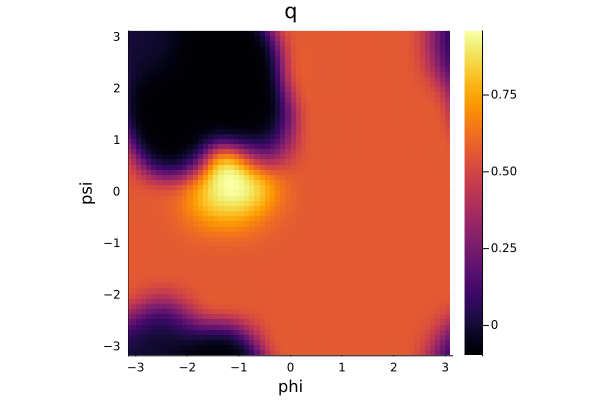}
	\includegraphics[width=0.495\textwidth]{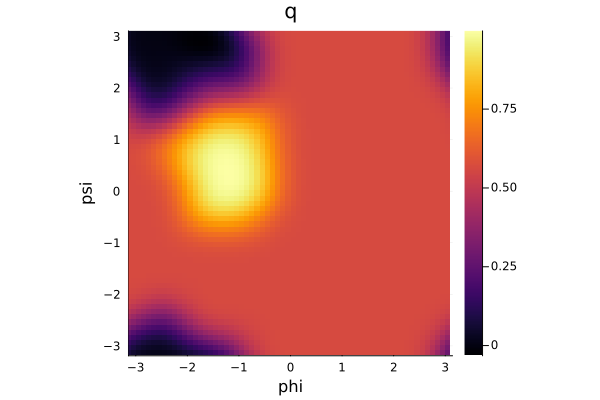}
\caption{{Committor approximations for the API based on the high friction approximation (\ref{adpAPI})--(\ref{uld+u}) of the underdamped Langevin dynamics for regularisation parameters $\epsilon\in\{0.002,\,0.05,\,0.1,\,2\}$ (from top left to bottom right).}}\label{fig:committorADP}
\end{figure}}

{Nevertheless we can only speculate here, since a detailed numerical analysis and a quantitative analysis of the high friction approximation of the control problem (that is not yet available) would go beyond the scope of this paper. We leave this interesting question for future research.}

{\begin{remark}
	We conjecture that the optimal control representation of the committor equation based on the risk-sensitive second moment formulation will show similar features when applied to underdamped Langevin dynamics. In both cases, the optimal control is proportional to the gradient of the log committor in the momentum variables that, as we have argued, can be approximated by the gradient of the committor in the position variables in the high friction regime. Since, moreover, the dominant eigenvalues of the underdamped Langevin operator, that determine the necessary regularisation that guarantees monotonicity of API in the risk sensitive formulation, converges to the dominant eigenvalues of the overdamped operator, we believe that the API algorithm for the risk-sensitive formulation of phase space molecular dynamics will face similar issues as in the log transformation case, possibly with an even stronger dependence of the resulting committor on the regularization since the second-moment formulation generally requires stronger regularization than the log transformation formulation.   
\end{remark}}

\subsection{Application to exit time problems}\label{sec:exit}

While both SOC representations of the rare event simulation problem, share the zero variance property, they are found to have reverse effects in terms of the likelihood of rare events; cf.~Example \ref{ex:2wellCommittor}. Specifically, we have found that in the log transformation based representation the likelihood of a rare event is \emph{increased}, whereas in the square root based transformation, the likelihood of the rare event is \emph{decreased}. For sampling problems that involve unbounded random stopping times and thus sample trajectories of indefinite length, the second formulation can lead to counter-intuitive behavior of importance sampling schemes, because it can lead to a massive  increase of the total computational cost, despite variance reduction. 

\begin{example}[First exit times] 
An extreme case is the first exit time from a set. Let $\tau$ be the first exit time from the (open and bounded) set $D\subset R^d$. This is our random variable of interest. In the log transformation setting, with $f=1$ and $g=0$, the value function equals the scaled log MGF of $\tau$:
\begin{equation}\label{tauCGF}
-\lambda^{-1}\log\E^x\!\left[e^{-\lambda\tau}\right] = \min_{u\in\cA}\E^x\!\left[ \tau^u + \frac{1}{2\lambda} \int_0^{\tau^u} |u_{s}|^{2}ds\right].
\end{equation}
Clearly, since $u=0$ is an admissible control, $\E^x[\tau^{u^*}]\le \E^x[\tau]$ must hold under the optimal control $u^*$.  As a consequence, the optimal control minimizes the variance \emph{and} reduces the average length of the sample trajectories. Note that under the hypothetical optimal control (i.e., ignoring discretization errors) that leads to a zero variance estimator, only one trajectory needs to be computed. Then, given the log moment generating function (\ref{tauCGF}), moments of $\tau$ can be extracted by differentiation \wrt the parameter $\lambda$, e.g.
\begin{equation}
\E^x[\tau] = -\frac{d}{d\lambda}\left. \log\E^x\!\left[e^{-\lambda\tau}\right]\right|_{\lambda=0}
\end{equation} 
or, equivalently,
\begin{equation}\label{limitCGF}
\E^x[\tau] = -\lim_{\lambda\searrow 0}\lambda^{-1}\log\E^x\!\left[e^{-\lambda\tau}\right]\,.
\end{equation}

On the other hand, the SOC formulation based on the second moment minimization leads to an optimal control with control law (cf.~\cite{awad2013})
\begin{equation}
c^*(x) = \sigma^\top\nabla\log\E^x[\tau]\,,
\end{equation}
which is pointing towards the interior of the set $D$ and which is diverging at the boundary (since $\E^x[\tau]$ is decreasing towards the boundary and is zero on $\partial D$). Hence, the optimal control is preventing the dynamics from ever leaving the set $D$, which leads to a zero variance importance sampling scheme with infinite run time, because $\tau^{u^*}=\infty$ with probability one. See \cite[Sec.~6.4.2]{anum} for details.
\end{example}

\paragraph{Computational aspects: mean first exit times}

The example shows that sampling mean first exit (or: passage) times or transition rates by importance sampling is not a trivial task: On the one hand, it seems that the log transform formulation based on the moment generating function is the method of choice because of the pathologies of the second moment minimization. On the other hand, computing moments requires evaluating the value function at $\lambda=0$, but the smaller $\lambda$ the stronger the penalization of the control; for $\lambda\to 0$, the optimal control becomes $u^*=0$. 
If $\lambda>0$, one can choose between computing (a) the value function $v_1$ as a proxy for the scaled log MGF or (b) computing the MGF by reweighting (i.e. standard importance sampling). 

Assuming that a numerical scheme is used to approximate the optimal control or the value function for $\lambda>0$, the variance will not be exactly zero, and option (a) inevitably yields a biased approximation of the MGF (since standard Monte Carlo will produce an unbiased approximation of the log MGF); see also (\ref{biasedEstimateQ}). Numerical methods beyond the API algorithm of Section \ref{sec:API} include, e.g. neural network approximations of the HJB equation \cite{JentzenEtal2017,nuesken2021}, least-squares Monte Carlo for backward SDEs \cite{GobetTurkedjiev2016,chaosRESIM}, or stochastic optimization \cite{kerimkulov2025mirror,lie2021frechet}, to mention just a few alternatives. 

As for option (b), we could build an importance sampling estimator of the MGF using the equality
\begin{equation}\label{reweightMGF}
\E^x\!\left[e^{-\lambda \tau}\right] = \E^x\!\left[e^{-\lambda \tau^{\hat{u}}}\hat{\mathscr{L}}^{-1}_{\tau^{\hat{u}}}(X^{\hat{u}})\right] 
\end{equation}  
with 
\begin{equation}
\hat{\mathscr{L}}_t = \exp\!\left(-\frac{1}{2}\int_{0}^{t}\vert \hat{u}_s\vert^2 ds + \int_{0}^{t} \hat{u}_s dB_s \right)
\end{equation}
being a numerical approximation of the optimal likelihood ratio. Such an estimator, however, will be very sensitive to the approximation of the optimal control and, in high dimensions, become essentially useless when the approximations are not close-to-perfect \cite{suboptimalIS}.  Generally, using option (b) is not recommended for high-dimensional problems \cite{agapiou2015importance,bengtsson2005curse}. 

We describe yet another alternative to compute mean first exit times.

\paragraph{From importance sampling to control variates}

It turns out that the even though the log transformation based SOC problem has a trivial solution in the limit $\lambda\to 0$, the optimal importance sampling estimator has a nontrivial limit and retains its zero variance property, even though the average length of the trajectories is not reduced as the control goes to zero. Indeed, combining (\ref{limitCGF})  with (\ref{reweightMGF}), it follows by dominated convergence that  
\begin{equation}
	\E^x[\tau] = \E^x\!\left[\tau^0 +\lim_{\lambda\searrow 0}\lambda^{-1}\log\hat{\mathscr{L}}_{\tau^{\hat{u}}}(X^{\hat{u}}) \right]
\end{equation}
where $\tau^0=\tau$ is the exit time of the uncontrolled dynamics. The limit expression inside the expectation converges to a zero-mean random variable; it acts as a control variate and annihilates the variance. 
The next theorem that has been proved in \cite{hartmann2024riskneutral} formalizes this observation; it can be seen as a risk-neutral limit of the log transformation based certainty-equivalence principle.   

\begin{theorem}\label{thm:is2cv}
    Let $\mathscr{L}^*$ be the likelihood ratio associated with the change of measure from the reference probability measure $P$ to the zero-variance probability measure $P^{u^*}$ according to Lemma \ref{lem:socGibbs}. 
Then, with probability one, 
\begin{equation}\label{limitCGF2}
    - \lim_{\lambda\searrow 0}\frac{1}{\lambda}\log\E^x\!\left[e^{-\lambda W(X^{u^*})-\log \mathscr{L}_{\tau^{u^*}}^*}\right] = W(X) - M_\tau(X),
\end{equation}
where $X$ solves the uncontrolled SDE (\ref{sde0}) with initial value $X_0=x$, and 
\begin{equation}
    M_t(X) = \int_0^t \sigma^T\nabla\Psi(X_s)\cdot dB_s
\end{equation}
is a martingale with the property 
\begin{equation}
\E[M_\tau(X)] = 0\,,\quad {\rm Var}(W(X) + M_\tau(X)) = 0\,.
\end{equation}
The function $\Psi$ is the solution of the linear boundary value problem (\ref{linBVP}) with $F=0$, $H=f$ and $G=g$. As a consequence, the right-hand side of (\ref{limitCGF2}) is an unbiased zero-variance estimator of $\E^x[W(X)]$.
\end{theorem}

\begin{proof}[Sketch of proof]
	Using that 
	\begin{equation}
	\mathscr{L}^*_t = \exp\!\left(-\frac{1}{2}\int_{0}^{t}\vert u^*_s\vert^2 ds + \int_{0}^{t} u^*_s dB_s \right),
\end{equation}
with $u^*_t = -\lambda\sigma^\top \nabla v_1(X_t^{u^*})$ where $v_1=v_1^\lambda$ is the solution to (\ref{PDE1}), and letting $\lambda\searrow 0$ in (\ref{PDE1}), we observe that:
\begin{itemize}
	\item[(a)] the nonlinear HJB equation (\ref{PDE1}) turns into the linear  boundary value problem (\ref{linBVP}) with coefficients $F=0$, $G=g$ and $H=f$ 
	\item[(b)] $v_1^\lambda\to\Psi$ and $\nabla v_1^\lambda\to \nabla\Psi$ where $\Psi$ solves the limit equation (\ref{linBVP}) 
	\item[(c)] the control $u^*$ converges to zero at rate $\lambda$ and $\lambda^{-1}\log\mathscr{L}$ converges to the martingale $M$ with the integrand $\nabla\Psi$.
\end{itemize}
The zero-variance property follows from It\^o's formula, using the fact that $\Psi$ solves the linear boundary value problem (\ref{linBVP}):  
\begin{equation}
\Psi(X_\tau) - \Psi(x) = \int_0^\tau \cL\Psi(X_s)\,ds + \int_0^\tau \sigma^T\nabla\Psi(X_s)\cdot dB_s\,,\quad x\in D
\end{equation}
where $\Psi(X_\tau)=g(X_\tau)$ and $ \cL\Psi(X_s)=-f(X_s)$ for $s<\tau$. Hence, almost surely,  
\begin{equation}
\Psi(x) = \underbrace{\int_0^\tau f(X_t)\,dt + g(X_\tau)}_{=W(X)} - \underbrace{\int_0^\tau \sigma^T\nabla\Psi(X_t)\cdot dB_t}_{=M_\tau(X)}\,.
\end{equation}
Since $\E^x\big[\exp\big(-\lambda W(X^{u^*})-\log \mathscr{L}_{\tau^{u^*}}^*\big)\big]=\E^x\big[\exp(-\lambda W(X))\big]$ by definition of $\mathscr{L}^*$, the left-hand side in (\ref{limitCGF2}) converges to $\E^x[W(X)]$, which concludes the sketch of the proof. For details we refer to the arXiv version of \cite{hartmann2024riskneutral}. 
\end{proof}

In contrast to the importance sampling estimator with reweighting, the estimators based on (\ref{limitCGF2}) are relatively robust under bad approximations of the control variate term $M_\tau$. This is illustrated in the next example. 

\begin{example}[First exit times, cont'd]\label{ex:mfetOU}
	We consider the exit problem for a reversible Ornstein-Uhlenbeck (OU) process 
	\begin{equation}
	dX_t = -AX_t\,dt + \sqrt{2\beta^{-1}}dB_t\,,\quad X_0=x
	\end{equation}
	in dimension $d=100$, where
	\begin{equation}
		A = \begin{pmatrix}
		2 & -1 & 0 & \ldots & 0 \\ -1 &  2  & -1 & \ldots & 0\\ 0 & \ddots & \ddots & \ddots & 0\\ \vdots & \vdots & \ddots & \ddots  & -1 \\ 0 & \ldots & 0 & -1 & 2  
	\end{pmatrix}\in\R^{d\times d}\,.
	\end{equation}
	We consider the exit from the set $D=\{x\in\R^d\colon |x|<R\}$ and  compute the mean first exit time (MFET) $\Psi(x) = \E^{x}[\tau]$ for the process starting from $x=0$. Since the exact control variate is not available (since we do not know $\Psi$), we replace it by the  approximation 
	\[
	\Phi(x)= \frac{1}{d}\frac{R^2 - |x|^2}{2\beta^{-1}}
	\]
	that is valid in the radially symmetric case if $A$ is similar to a scalar multiple of the identity and $d\to\infty$; see \cite{sphericalOU}. Using $\Phi$ instead of $\Psi$ as integrator, our suboptimal control variate estimator becomes
	\begin{equation}
		\hat{\Psi}_N(x) = \frac{1}{N}\sum_{i=1}^N\left(\tau_i + \sqrt{2\beta^{-1}}\int_0^{\tau_i} \nabla\Phi(X_{t,i})\cdot dB_{t,i}\right)\,,
	\end{equation}
	where the sum is over $N$ independent realizations of the process $X$.
	Figure \ref{fig:exitOU} shows Monte-Carlo estimates of the MFET for $\beta=10$ for a crude Monte Carlo approximation (orange curve) and the asymptotic control variate approximation $\hat{\Psi}_N$ (green curve), each for a sample size $N=10$; for comparison, the Figure also shows the reference Monte Carlo approximation for $N=1000$ (blue curve).

\begin{figure}
	\centering
	\includegraphics[width=0.65\textwidth]{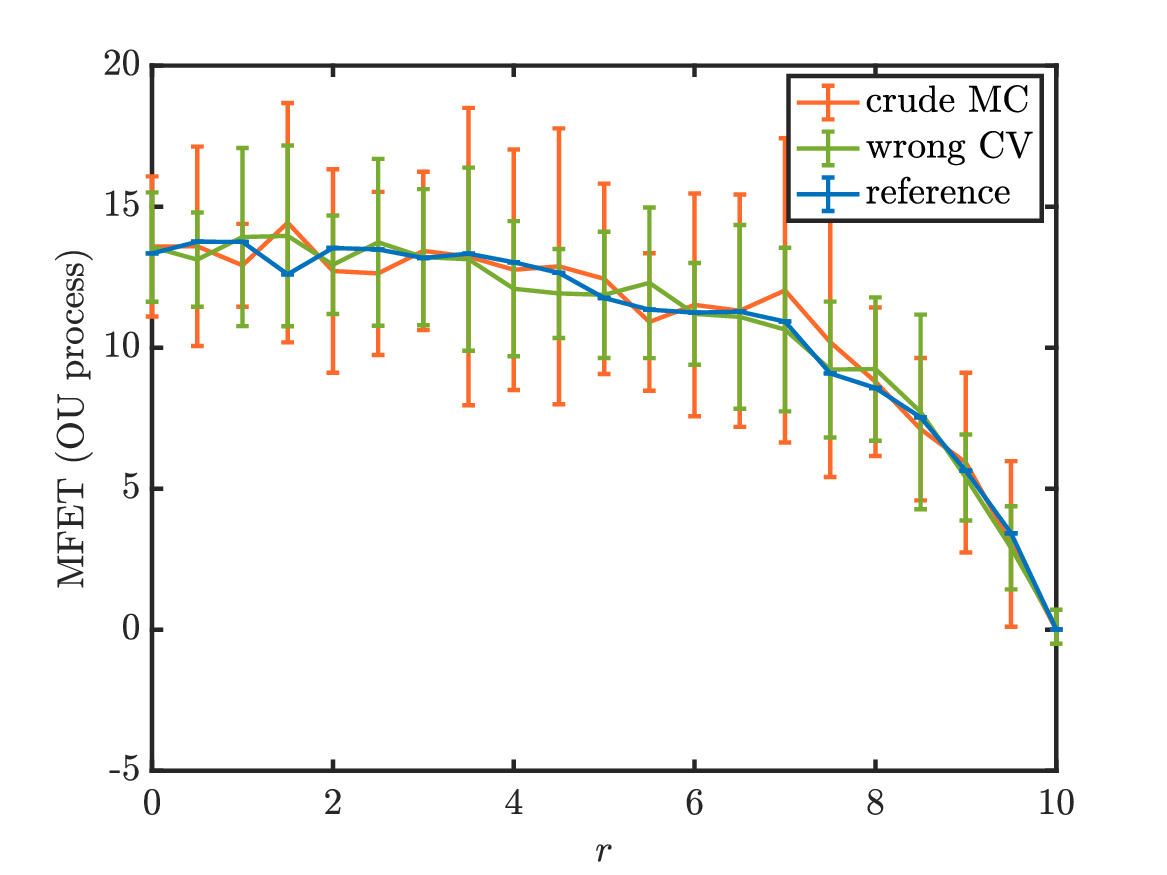}
	\caption{MFET estimates in dimension $d=100$ and their 95\% confidence intervals: crude Monte Carlo (orange) and control variates (green) with suboptimal function $\tilde{\Psi}$ for $N=10$, and the reference  Monte Carlo approximation (blue) for $N=1000$. All simulations were done using an Euler-Maruyama discretization of the SDE and the martingale with steps size $\Delta t=10^{-3}$.}\label{fig:exitOU}
\end{figure}
\end{example}

We emphasize that the robustness of the control variate estimator under suboptimal controls that has been empirically observed in Example \ref{ex:mfetOU} is in stark contrast to the brittleness of importance sampling in high dimensions \cite{bengtsson2005curse,agapiou2015importance}. The suboptimality issue is especially important if the length of the trajectories used for sampling is not bounded \cite{suboptimalIS}. 
Yet, we believe that the findings are not confined to exit time computations. In practice, a numerical approximation of the control variate will often be obtained from solving a simplified, lower-dimensional equation or surrogate model, and first steps towards understanding the properties of suboptimal control variates have been undertaken in \cite{newton1994variance}; see also ~\cite{belomestny2024,roussel2019,oates2023}.

\section{Conclusions}\label{sec:fin}

We have studied importance sampling for rare events from the perspective of certainty-equivalent expectations. Specifically, two different representations of the optimal importance sampling distribution for an SDE that both lead to zero variance estimators for the quantity of interest have been looked at in detail; both are based on nonlinear transformations of the random variable of interest that can be interpreted as certainty-equivalent control problem: a representation based on a logarithmic transformation of the moment generating function (MGF) of the random variable of interest, and another one based on a square root transformation of the second moment. 

For both representations, we have devised approximate policy iteration (API) schemes and analyzed their convergence. Both schemes are monotonic in terms of the cost functional and converge to the correct optimal control, where the square root representation requires some regularization to enforce convergence. For a high dimensional committor problem with spherical symmetry, we have tested both representations and associated API schemes numerically and confirmed empirically that they converge to a biased estimate of the committor function and the optimal control that generates the optimal importance measure. 

Finally, we have discussed {the application of API to phase space molecular dynamics and} the notoriously difficult exit problem that is relevant for the computation of mean first passage times and transition rates. {We have sketched what is necessary to integrate the API algorithm into phase space molecular dynamics models, using the computation of the committor function for alanine dipeptide (ADP) in vaccum in the high friction regime as a test case. For exit problems, we have moreover} demonstrated that naive minimization of the second moment leads to estimators that formally have zero variance, but generate infinitely long trajectories with probability one. The log transformation based formulation can be used here to devise robust low-variance estimators for mean first exit or passage times. These estimators are, however, not importance sampling estimators, but control variates. In particular, they are based on uncontrolled dynamics. The robustness of the control variate scheme for mean first exit times has been demonstrated numerically for a high-dimensional Ornstein-Uhlenbeck process. 

The numerical examples, though high-dimensional, are admittedly simple, but they reveal the key features of the stochastic control formulations and numerical algorithms. Future work ought to address the question of suboptimality of the corresponding statistical estimators when reduced-order or surrogate models are used. {As we have seen in case of the committor problem for ADP,} this is especially relevant for molecular dynamics applications that can be very high-dimensional, but often admit low-dimensional representations in terms of suitable feature variables or reaction coordinates.

\section*{Acknowledgement}

This work was supported by the DFG Collaborative Research Center 1114 ``Scaling
Cascades in Complex Systems'', project no. 235221301, projects A05 ``Probing Scales in Equilibrated Systems by Optimal Nonequilibrium Forcing'' and B03 ``Multilevel coarse graining of multiscale problems''. {The authors thank Wei Zhang for helpful remarks about the high friction approximation of the committor equation.}

\appendix

\section{Generalized stochastic optimal control}\label{sec:genSOC}

In this section we state a general SOC problem on an indefinite time horizon, special forms of which are extensively used throughout the paper. For details regarding the regularity requirements on the coefficients of the cost functional, we refer to \cite[Sec.~3]{Pham} and references therein.

\begin{definition}[Generalized SOC problem]\label{def:generalSOC}
	Let 
\begin{equation}	
	J(t,x,u)=\E^{t,x}\!\left[\int_{t}^{T} f(r,X^u_r,u_r)\Gamma(r)dr+g(T,X^u_T)\Gamma(T)\right],	
\end{equation} 
be an objective function where 
 $\Gamma(s):=exp\!\left(-\int_{t}^{s}\beta(r,X^u_r,u_r)dr\right)$, with $f,g,\beta$ being measurable functions (with conditions added in the course of this section). We define the generalized SOC problem on a finite time horizon by
	\begin{equation}
		\begin{split}
	&\inf_{u\in \cU} J(t,x,u)\\
	&\text{s.t.}~ 	dX^u_s=(b(X^u_s)+\sigma(X^u_s)u_s)ds+\sigma(X^u_s)dB_s,\quad X^u_t=x	,
	\end{split}
	\end{equation}
with value function $v(t,x):=\min\limits_{u} J(t,x,u)$. 
\end{definition}

\begin{lemma}[It{\^o}'s formula]\label{Ito}
	Let $a\in \cA$ be fixed. It{\^o}'s formula applied to 
	\begin{equation}\Gamma(t+h)v(t+h,X^a_{t+h})\end{equation} yields 
	\begin{align*}
		& \Gamma(t+h)v(t+h,X^a_{t+h})
		=v(t,x)\\
		& +\int_{t}^{t+h}\Gamma(r)
		\left\lbrace
		-\beta(r,X^a_r,a)v(r,X^a_r)+\frac{\partial}{\partial r}v(r,X^a_r)+\cL^av(r,X^a_r)
		\right\rbrace dr  \\
		& +
		\int_{t}^{t+h} 	\Gamma(r)\frac{\partial}{\partial x} \sigma(X^a_r) dB_r.
	\end{align*}
	\end{lemma}
	
	\begin{proof}
		We define $h(s,x):=\Gamma(s)v(s,x)$ and $Y_s:=s$. Using the shorthands $h_y=\frac{\partial h}{\partial y}$, $h_x=\frac{\partial h}{\partial x}$, etc., we have 
		\begin{align*}
			dh(s,X^a_{s})&=dh(Y_s,X^a_{s})\\
			&=h_y(Y_s,X^a_{s})dY_s+h_x(Y_s,X^a_{s})dX^a_s\\
			& \quad +h_{xy}(Y_s,X^a_{s})d\langle X^a,Y \rangle_s
			+\frac{1}{2}h_{yy}(Y_s,X^a_{s})d\langle Y\rangle_s\\
			& \quad  +\frac{1}{2}h_{xx}(Y_s,X^a_{s})d\langle X^a\rangle_s\\
			&=h_y(Y_s,X^a_{s})ds+h_x(Y_s,X^a_{s})dX^a_s+\frac{1}{2}h_{xx}(Y_s,X^a_{s})\sigma^2(X^a_s)ds \\
			&=\left(\Gamma'(s)v(s,X^a_{s})+\Gamma(s)\frac{\partial}{\partial s}v(s,X^a_{s}) \right)ds\\
			&\quad +\Gamma(s)\frac{\partial}{\partial x}v(s,X^a_{s})\left[b(X^a_s)+\sigma(X^a_s)a~ ds+\sigma(X^a_s)dB_s \right]\\
			&\quad +\frac{1}{2}  \Gamma(s) \frac{\partial^2}{\partial x^2} v(s,X^a_{s})\sigma^2(X^a_s)ds\\
			&=\Gamma(s)  \biggl\{
			\left[ -\beta(s,X^a_s,a)v(s,X^a_{s})+\frac{\partial}{\partial s}v(s,X^a_{s})+\cL^av(s,X^a_{s})\right]ds\\ 
			&\quad \quad  \quad \quad +\frac{\partial}{\partial x}v(s,X^a_{s})\sigma(X^a_s)dB_s\biggr\},
		\end{align*}
        where we have used $dY_s=ds$, $d\langle Y\rangle_s=0$, $d\langle X^a,Y \rangle_s=0$, $d\langle X^a\rangle_s=\sigma^2(X^a_s)ds$, $\Gamma(t)=1$, $v(t,X^a_t)=v(t,x)$ and $\Gamma'(s)=-\beta(s,X_s^a,a)\Gamma(s)$.
		Hence,
		\begin{align*}
			& \Gamma(t+h)v(t+h,X^a_{t+h})
			=v(t,x)\\
			&\quad +\int_{t}^{t+h}\Gamma(r)
			\left\lbrace
			-\beta(r,X^a_r,a)v(r,X^a_r)+\frac{\partial}{\partial r}v(r,X^a_r)+\cL^av(r,X^a_r)
			\right\rbrace dr  \\
			&\quad+
			\int_{t}^{t+h} 	\Gamma(r)\frac{\partial}{\partial x} \sigma(X^a_r) dB_r.
		\end{align*}
	\end{proof}

\subsection{Indefinite time horizon}

Next we consider the objective function on a random time horizon
\begin{equation}	
	J(t,x,u)=\E^{t,x}\left(\int_{t}^{\tau} f(r,X^u_r,u_r)\Gamma(r)dr+g(\tau,X^u_\tau)\Gamma(\tau)\right),	
\end{equation}
where for simplicity $\tau$ denotes the first exit time of the process $X^u$ from some set $D$, starting from time $t$. 

\begin{lemma}[Dynamic programming principle; cf.~\cite{Pham}]
	The value function $v$ satisfies, 
	\begin{equation}\label{ineq_Pham}
		v(t,x)\leq \E^{t,x}\!\left[\int_{t}^{t+h} f(r,X^u_r,u_r)\Gamma(r)dr+v(t+h,X^u_{t+h})\right],
	\end{equation}
and every $k>t$, we have
	\begin{equation}
		v(t,x)=\inf_{u\in \cU} \E^{t,x}\!\left[\int_{t}^{k\wedge \tau} f(r,X^u_r,u_r)dr+v(k\wedge \tau,X^u_{k\wedge \tau})\right].
	\end{equation}	

\end{lemma}

	\begin{proof}
		By the tower property of conditional expectations, 
		\begin{align*}
			& \E^{t,x}\!\left[J(k\wedge\tau,X^u_{k\wedge \tau},u)  \right]=\E^{t,x}\!\left[\E^{X^u_{k\wedge \tau}}\!\left[\int_{t}^{\tau}f(r,X^u_r,u_r)dr+g(\tau,X^u_{\tau}) \right] \right]\\
			&=\E^{t,x}\!\left[\E^{t,x}\!\left[\int_{k\wedge \tau}^{\tau}f(r,X^u_r,u_r)dr+g(\tau,X^u_{\tau})~\vrule~ \cF_{k\wedge \tau}\right]\right]\\
			&=\E^{t,x}\!\left[\int_{t}^{\tau}f(r,X^u_r,u_r)dr -\int_{t}^{k\wedge \tau}f(r,X^u_r,u_r) dr +g(\tau,X^u_{\tau}) \right]\\
			&=J(t,x,u)-\E^{t,x}\!\left[\int_{t\wedge \tau}^{k\wedge \tau}f(r,X^u_r,u_r) dr \right].
		\end{align*}
		Hence,
		\begin{align*}
			J(t,x,u)&=	\E^{t,x}\!\left[J(k\wedge\tau,X^u_{k\wedge \tau},u)+\int_{t}^{k\wedge \tau}f(r,X^u_r,u_r) dr  \right]\\
			& \geq \E^{t,x}\!\left[v(k\wedge\tau,X^u_{k\wedge \tau})+\int_{t}^{k\wedge \tau}f(r,X^u_r,u_r) dr \right],	
		\end{align*}
		and 
		\begin{equation}\label{ineq1}
			\begin{split}
			v(t,x)&=\inf_{u\in \cU}J(t,x,u)\\
			&\geq \inf_{u\in \cU} \E^{t,x}\!\left[v(k\wedge\tau,X^u_{k\wedge \tau})+\int_{t}^{k\wedge \tau}f(r,X^u_r,u_r) dr  \right].	
			\end{split}	
		\end{equation}
		Let $u$ and $\tau$ be fixed but arbitrary, then for every $\epsilon >0$ and $\omega \in \Omega$, there is an admissible control $u^{\epsilon}$ such that
		\begin{equation}\label{inf}
			v(k\wedge \tau(\omega),X^u_{k\wedge \tau(\omega)})+\epsilon \geq J(k\wedge \tau(\omega),X^u_{k\wedge \tau(\omega)},u^{\epsilon}),
		\end{equation}	
		by definition of the value function and the infimum. 
		We now define the process
		\begin{equation}
			\hat{u}_s(\omega):=
			\begin{cases}
				&u_s(\omega),\quad s\in [0,k\wedge\tau(\omega))\\
				&u_s^{\epsilon}(\omega),\quad s\in [k\wedge \tau(\omega),\tau(\omega)],
			\end{cases}	
		\end{equation}
		then
		\begin{align*}
			v(t,x)\leq J(t,x,\hat{u})&=\E^{t,x}\!\left[
			\int_{t}^{k \wedge \tau} f(r,X^{\hat{u}}_r,\hat{u}_r)dr+J(k\wedge t,X^{\hat{u}}_{k\wedge \tau},\hat{u})
			\right]\\
			&=\E^{t,x}\!\left[
			\int_{t}^{k \wedge \tau} f(r,X^{u}_r,u_r)dr+J(k\wedge t,X^{u^{\epsilon}}_{k\wedge \tau},u^{\epsilon})
			\right]\\
			&\leq  \E^{t,x}\!\left[
			\int_{t }^{k \wedge \tau} f(r,X^{u}_r,u_r)dr+v(k\wedge t,X^{u}_{k\wedge \tau})
			\right]+\epsilon,
		\end{align*}
		thus
		\begin{align*}
			v(t,x)	\leq \inf_{u\in \cU} \E^{t,x}\!\left[
			\int_{t}^{k \wedge \tau} f(r,X^{u}_r,u_r)dr+v(k\wedge t,X^{u}_{k\wedge \tau})
			\right]+\epsilon,
		\end{align*}
		and with $\epsilon \rightarrow 0$, we obtain
		\begin{align*}
			v(t,x)
			& \leq \inf_{u\in \cU} \E^{t\wedge\tau,x}\!\left[
			\int_{t \wedge \tau}^{k \wedge \tau} f(r,X^{u}_r,u_r)dr+v(k\wedge t,X^{u}_{k\wedge \tau})
			\right].
		\end{align*}
	Together with equation \eqref{ineq1}, this gives
		\begin{equation}
			v(t,x)=\inf_{u\in \cU} \E^{t,x}\!\left[
			\int_{t }^{k \wedge \tau} f(r,X^{u}_r,u_r)dr+v(k\wedge t,X^{u}_{k\wedge \tau})
			\right].
		\end{equation}
		The proves the dynamic programming principle. 
	\end{proof}


\subsection{Dynamic programming equations}

We formally derive the HJB equation corresponding to the generalized SOC problem of Definition \ref{def:generalSOC}. For details regarding the regularity of coefficients that guarantee existence of classical solutions, we refer to the relevant literature on nonlinear partial differential equations, e.g.~\cite{evans2022partial,flemingrishel,fleming2006,gilbarg2001elliptic}.

\begin{theorem}[HJB equations]
		Let $D\subset \R^d$ be an open set with smooth boundary $\partial D$. The HJB equation for the finite time horizon is given by
		\begin{equation}
            \begin{aligned}
			0 = & \min_{a\in \cA}\Bigg \{ f(t,x,a)+\frac{\partial}{\partial t}v(t,x)-\beta(t,x,a)\cdot v(t,x)\\ & + b(t,x,a)\cdot \nabla v(x)
			+\frac{1}{2}\sigma(x)\sigma(x)^\top:\nabla_{xx}^2v(x)\Bigg\},~ x\in D, t<T\\ 
			v(T,x) = & g(T,x),~ x\in \partial D\,.
		      \end{aligned}
        \end{equation}    
        For the random time horizon, it reads
		\begin{equation}
            \begin{aligned}
			0 = & \min_{a\in \cA}\Bigg \{ f(t,x,a) +\frac{\partial}{\partial t}v(t,x) - \beta(t,x,a)\cdot v(t,x) \\ & +b(t,x,a)\cdot \nabla v(x)  
			 +\frac{1}{2}\sigma(x)\sigma(x)^\top:\nabla_{xx}^2v(x)\Bigg\}, ~ x\in D,~t<\tau\\
			v(t,x) = & g(t,x), ~ t>0, x \in \partial D.
		     \end{aligned}
        \end{equation}
		\end{theorem}
	\begin{proof}
		By Lemma \ref{Ito}, 
		\begin{align*}
			& v(t+h,X^a_{t+h})
			=	\Gamma^{-1}(t+h)v(t,x)\nonumber\\
			& +	\Gamma^{-1}(t+h)\int_{t}^{t+h}\Gamma(r)
			\left\lbrace
			-\beta(r,X^a_r,a)v(r,X^a_r)+\frac{\partial}{\partial r}v(r,X^a_r)+\cL^av(r,X^a_r)
			\right\rbrace dr \nonumber \\
			& +	\Gamma^{-1}(t+h)
			\int_{t}^{t+h} 	\Gamma(r)\frac{\partial}{\partial x} \sigma(X^a_r) dB_r.
		\end{align*}
		Due to equation \eqref{ineq_Pham},
		we obtain by replacing $v(t+h,X^a_{t+h})$ and using It{\^o} formula (Lemma \ref{Ito}) the following after taking expectations: 
		\begin{align*}
			v(t,x)\leq & \E^{t,x}\!\left[\int_{t}^{t+h} f(r,X^u_r,a)\Gamma(r)dr+v(t+h,X^a_{t+h})\right]\\
			= &	\E^{t,x}\!\left[\int_{t}^{t+h} f(r,X^u_r,a)\Gamma(r)dr
			+\Gamma^{-1}(t+h)v(t,x) \right]\\
			& +	\E^{t,x}\!\left[\Gamma^{-1}(t+h)\int_{t}^{t+h}\Gamma(r)
			\left\lbrace
			-\beta(\cdot)v(\cdot)+\frac{\partial}{\partial r}v(\cdot)+\cL^a v(\cdot)
			\right\rbrace dr \right].
		\end{align*}
		Hence,
		\begin{align*}
			v(t,x)&\leq\E^{t,x}\!\left[\int_{t}^{t+h} f(\cdot)\Gamma(\cdot)dr
			+v(t,x)\Gamma^{-1}(t+h) \right]\\
			&\quad+	\E^{t,x}\!\left[\Gamma^{-1}(t+h)\int_{t}^{t+h}\Gamma(r)
			\left\lbrace
			-\beta(\cdot)v(\cdot)+\frac{\partial}{\partial r}v(\cdot)+\cL^av(\cdot)
			\right\rbrace dr \right],		
		\end{align*}
		and thus
		\begin{align*}
			0&\leq \E^{t,x}\!\left[\int_{t}^{t+h}f(\cdot)\Gamma(\cdot)dr+v(t,x)[\Gamma^{-1}(t+h)-1] \right]\\
			&\quad +\E^{t,x}\!\left[\Gamma^{-1}(t+h)\int_{t}^{t+h}\Gamma(r)[-\beta(\cdot)v(\cdot)+\frac{\partial}{\partial r}v(\cdot)+\cL^av(\cdot)]dr \right].
		\end{align*}
		Upon multiplication with $\frac{1}{h}$, we obtain after taking the limit $h\rightarrow 0$,\footnote{With $\lim\limits_{h\rightarrow 0}
			\frac{\widehat{f}(h)\cdot \widehat{g}(h)}{h} =\lim\limits_{h\rightarrow 0} \widehat{f}(h)\cdot \lim\limits_{h\rightarrow 0}\frac{1}{h}\widehat{g}(h)$ assuming that  both sequences converge.}
		\begin{align*}
			0&\leq f(t,x,a)\Gamma(t)+\E^{t,x}\!\left[ v(t,x)\frac{d}{dt}\Gamma^{-1}(t)\right]\\
			& +\E^{t,x}\!\left[\lim\limits_{h\rightarrow 0}\Gamma^{-1}(t+h)\cdot \lim\limits_{h\rightarrow 0}\frac{1}{h}\int_{t}^{t+h}\Gamma(r)\left\{-\beta(\cdot)v(\cdot)+\frac{\partial}{\partial r}v(\cdot)+\cL^av(\cdot) \right\}dr \right]\\
			&=f(t,X_t^a,a)+\E^{t,x}\!\left[ \lim\limits_{h\rightarrow 0}\frac{1}{h}\int_{t}^{t+h}\Gamma(r)\left\{-\beta(\cdot)v(\cdot)+\frac{\partial}{\partial r}v(\cdot)+\cL^av(\cdot) \right\}dr\right]\\
			&= f(t,x,a)+\E^{t,x}\!\left[\Gamma(t)\left\{-\beta(t,X^a_t,a)v(t,X^a_t)+\frac{\partial}{\partial t}v(t,X^a_t)+\cL^av(t,X^a_t) \right\}\right]\\
			&=f(t,x,a)-\beta(t,X^a_t,a)v(t,X^a_t)+\frac{\partial}{\partial t}v(t,X^a_t)+\cL^av(t,X^a_t)\\
			&=f(t,x,a)-\beta(t,x,a)v(t,x)+\frac{\partial}{\partial t}v(t,x)+\cL^av(t,x),
		\end{align*}
		as $\Gamma^{-1}(t)=1$, $\frac{d}{dt}\Gamma^{-1}(t)=0$ and $\lim\limits_{h\rightarrow 0}\Gamma^{-1}(t+h)=1$.
		As 
		\begin{equation}
			\begin{split}	
				v(t,x) & =\inf\limits_{u\in \cU}\E^{t,x}\!\left[\int_{t}^{T} f(r,X^u_r,u_r)\Gamma(r)dr+g(T,X^u_T)\Gamma(T)\right]\\
				&=\E^{t,x}\!\left[\int_{t}^{T} f(r,X^{u^{\ast}}_r,u^{\ast}_r)\Gamma(r)dr+g(T,X^{u^{\ast}}_T)\Gamma(T)\right],
			\end{split}
		\end{equation}
		where $u^{\ast}$ denotes the optimal control, we obtain equality using Feynman-Kac formula, i.e.\ 
		\begin{equation}
			f(t,x,u^{\ast}_t)-\beta(t,x,u^{\ast}_t)v(t,x)+\frac{\partial}{\partial t}v(t,x)+\cL^{u^{\ast}}v(t,x)=0.	
		\end{equation}
		This yields the assertion. 
	\end{proof}

	\begin{remark}
	If the SDE coefficients as well as $f,g,\beta$ are time-homogeneous, then so is $v=v(x)$, and the HJB backward evolution equation reduces to a boundary value problem.    
	\end{remark}

{
\section{High friction limit of the committor equation}\label{app:hfLimit}}

{In this section, we give a formal justification of the control approximation \ref{hfControl} in the high friction regime. To this end, we consider the underdamped Langevin equation with (diffusive) high friction scaling, 
\begin{equation}\label{uld-eps}
\begin{aligned}
	dR_t & = \frac{1}{\eps}M^{-1}S_tdt\\
	dS_t & = -\frac{1}{\eps}\nabla V(R_t) dt - \frac{1}{\eps^2}\gamma M^{-1}S_t dt + \frac{1}{\eps}\sqrt{2\beta^{-1}\gamma}\, dW_t \,,
\end{aligned}
\end{equation}
where $\eps>0$ is some small parameter, $\beta=(k_B T)^{-1}$ is a positive constant, $M,\gamma\in\R^{n\times n}$ are the symmetric and positive definite mass and friction matrices, and $V\colon\R^n\to \R$ is the smooth molecular interaction potential; we assume that $V\in C^2$ is bounded below and satisfies a quadratic growth condition.}
{\subsection{Formal asymptotic expansion}\label{app:fax}
We seek a formal asymptotic expansion of the committor equation
\begin{equation}\label{ceq-eps}
\begin{aligned}
	\cL^\eps q^\eps(r,s) & = 0\,, \quad  (r,s)\in \R^{2n}\setminus (\overline{\cA\cup\cB})\\
	q^\eps(r,s) & = 0\,,  \quad  (r,s)\in \cA\\
	q^\eps(r,s) & = 1\,,  \quad  (r,s)\in \cB\,,
\end{aligned}
\end{equation}
where $\cA:=A\times\R^n, \cB:=B\times\R^n$ for bounded and open sets $A,B\subset\R^n$. The operator $\cL^\eps$ is degenerate elliptic and of the form 
\[
	\cL^\eps = \frac{1}{\eps^2}\cL_{0} + \frac{1}{\eps}\cL_{1}
\]
where the $\cL_i$ are given by 
\begin{eqnarray*}
	\cL_0 & = \beta^{-1}\gamma\colon\nabla^2_s - \left(\gamma M^{-1}s\right)\cdot\nabla_s\\
	\cL_1 & = \left(M^{-1}s\right)\cdot\nabla_r - \nabla V(r)\cdot \nabla_s\,.
\end{eqnarray*}
We seek a perturbative expansion of the solution to (\ref{ceq-eps}) of the form
\[
q^\eps = q_0 + \eps q_1 + \eps^2 q_2 + \ldots\,. 
\]
Inserting the ansatz into (\ref{ceq-eps}) and equating equal powers of $\eps$, we obtain a hierarchy of equations the first three of which read
\begin{subequations}\label{epsHierarchy}
	\begin{align}\label{eps-1}
		\cL_0 q_0 & = 0\\\label{eps0}
		\cL_0 q_1 & = -\cL_1 q_0\\\label{eps1}
		\cL_0 q_2 & = - \cL_1 q_1
	\end{align}
\end{subequations}
Since $\cL_0$ is a uniformly elliptic differential operator in $s$, the first equation in (\ref{epsHierarchy}) implies that $q_0=q_0(r)$ is independent of $s$. Therefore, and since (\ref{ceq-eps} holds for all $\eps>0$, the function $q_0$ must satisfy the boundary conditions 
\begin{align*}
	q_0(r) & = 0\,,\quad r\in A\\
	q_0(r) & = 1\,,\quad r\in B\,.
\end{align*}
By the Fredholm alternative, and since $\cL_0$ has a non-trivial nullspace, equation (\ref{eps0}) is solvable if and only if 
\[
 \int_{\R^n} \cL_1 q_0 \,\mu_0(ds) = 0\,,
\] 
where $\mu_0=\cN(0,\beta^{-1}M)$ is the unique invariant measure of the Ornstein-Uhlenbeck process generated by $\cL_0$. This, together with the left hand side of  (\ref{eps0}) implies that 
\[
q_1(r,s) = \left(\gamma^{-1}s\right)\cdot\nabla q_0(r) + \kappa
\]
where $\kappa\in\mathrm{ker}(\cL_0)$ plays no role in what follows and is henceforth ignored. Plugging the ansatz for $q_1$ into equation (\ref{epsHierarchy}), using again the Fredholm alternative and carrying out the Gaussian integral over $\mu_0$, we conclude that 
\begin{align*}
 	0 & = \int_{\R^n} \cL_1 q_1 \,\mu_0(ds)\\
	&  = \int_{\R^n} \left\{\left(M^{-1}s\right)\cdot \left(\nabla^2q_0(r)\gamma^{-1}s\right) - \nabla V(r)\cdot\left(\gamma^{-1}\nabla q_0(r)\right)\right\}\mu_0(ds)\\
	&  = \beta^{-1}\gamma^{-1}\colon\nabla^2q_0(r) - \left(\gamma^{-1}\nabla V(r)\right)\cdot\nabla q_0(r)\,,
\end{align*}
for all $r\in \R^{n}\setminus (\overline{A\cup B})$. 
As a consequence, $q_0$ solves the committor equation
\begin{equation}\label{ceq-0}
\begin{aligned}
	\overline{\cL} q(r) & = 0\,, \quad  r\in \R^{n}\setminus (\overline{A\cup B})\\
	q(r) & = 0\,,  \quad  r\in A\\
	q(r) & = 1\,,  \quad  r\in B\,,
\end{aligned}
\end{equation}
for the overdamped Langevin equation with infinitesimal generator 
\[
 \overline{\cL}  = \beta^{-1}\gamma^{-1}\colon\nabla^2_r - \left(\gamma^{-1}\nabla V(r)\right)\cdot\nabla_r\,.
\]
To find the first order expansion of the associated optimal control outside of a small neighbourhood of the non-target set $A$ so that both $q^\eps$ and $q_0$ are bounded away from 0, we expand the value function in powers of $\eps$:
\begin{align*}
	-\log q^\eps & = -\log\left(q_0 + \eps q_1 + \eps^2q_2 + \ldots\right)\\
	& = - \log q_0 - \log\left(1 + \eps \frac{q_1}{q_0} + \ldots \right)\\
	& = - \log q_0 - \eps \frac{q_1}{q_0} + \ldots\\
	& = - \log q_0 - \eps \left(\gamma^{-1}s\right)\cdot \nabla\log q_0  + \ldots
\end{align*}
For simplicity, we assume that the remainder terms that involve $q_2$ and that are formally of order $\eps^2$ are uniformly bounded in $r$, so that the remainder becomes negligible for sufficiently small $\eps$.} 

{
\subsection{Associated optimal control problem}
The optimal control problem associated with (\ref{uld-eps}) and the committor equation (\ref{ceq-eps}) is the minimisation of the cost functional 
\begin{equation}\label{committorCost}
	J^\eps(u) =  \E\left[\int\limits_{0}^{\tau^u} \frac{1}{2}\vert u_s\vert^2ds -\log \left(\mathbb{1}_{\partial B} \left(R^u_{\tau^u} \right)\right)\right]
\end{equation}
subject to the underdamped Langevin dynamics
\begin{equation}\label{uld+u-eps}
\begin{aligned}
	dR^u_t & = \frac{1}{\eps}M^{-1}S^u_tdt\\
	dS^u_t & = -\frac{1}{\eps}\nabla V(R^u_t) dr - \frac{1}{\eps^2}\gamma M^{-1}S^u_t dt + \frac{\varsigma}{\eps}\,u_t\,dt + \frac{\varsigma}{\eps}\, dW_t \,,
\end{aligned}
\end{equation} 
with the shorthand $\varsigma$ for the invertible $n\times n$-matrix $\sqrt{2\beta^{-1}\gamma}$.
As a consequence of Definition \ref{SOC1} together with Lemma \ref{lem:socGibbs} and the formal asymptotic expansion of Section \ref{app:fax}, we obtain the control approximation:
\begin{equation}\label{hfControl2}
	u^*_t = \frac{\sigma^\top}{\eps}\nabla\log q^\eps(R_t^u,S_t^u)\approx \varsigma^\top\gamma^{-1}\,\nabla\log q_0(R^u_t)\,.
\end{equation}
Using that $\varsigma=\sqrt{2\beta^{-1}\gamma}$ together with the approximation (\ref{hfControl2}), we obtain the suboptimally controlled underdamped Langevin equation
\begin{equation*}
	\begin{aligned}
		dR^*_t & = \frac{1}{\eps}M^{-1}S^*_tdt\\
	dS^*_t & = -\frac{1}{\eps}\nabla V(R^*_t) dr - \frac{1}{\eps^2}\gamma M^{-1}S^*_t dt + \frac{1}{\eps}2\beta^{-1}\,\nabla\log q_0(R^*_t)\,dt + \frac{\varsigma}{\eps}\, dW_t \,,
	\end{aligned}
\end{equation*}  
that in the limit $\eps\to 0$ converges weakly to its optimally controlled overdamped limit (cf.~\cite[App.~A]{hartmann2014optimal})
\begin{equation*}
		\gamma dR^*_t = -\nabla V(R^*_t) dr + 2\beta^{-1}\,\nabla\log q_0(R^*_t)\,dt + \sqrt{2\beta^{-1}\gamma}\, dW_t \,,
\end{equation*} 
that minimizes (\ref{committorCost}). 
}

\end{document}